%% file: main.tex
\documentclass{mathincs}[runningheads] 
\usepackage[T1]{fontenc}

\usepackage{amsthm}
\usepackage{amssymb}
\usepackage{amscd}
\usepackage{amsmath}
\usepackage[all]{xy}
\usepackage{algorithmic,algorithm}
\usepackage{here}
\usepackage{graphicx}
\usepackage{color} %
\usepackage{enumerate} %
\usepackage{listliketab}
\usepackage{multirow}
\usepackage{mathrsfs}


\makeatletter
  
  \@addtoreset{algorithm}{subsection}
\makeatother

\renewcommand{\figurename}{Fig.}
\renewcommand{\tablename}{Table }

\theoremstyle{definition}

\theoremstyle{definition}

\theoremstyle{plain}

\theoremstyle{plain}
\newtheorem{theor}{Theorem}

\theoremstyle{plain}

\theoremstyle{plain}

\theoremstyle{plain}
\newtheorem{thm}{Theorem}[subsection]

\theoremstyle{definition}
\newtheorem{ex}[thm]{Example}

\theoremstyle{definition}

\theoremstyle{definition}

\theoremstyle{definition}

\theoremstyle{definition}
\newtheorem{rem}[thm]{Remark}

\theoremstyle{plain}
\newtheorem{prop}[thm]{Proposition}

\theoremstyle{plain}
\newtheorem{lem}[thm]{Lemma}

\theoremstyle{plain}

\theoremstyle{definition}

\theoremstyle{definition}

\theoremstyle{definition}

\theoremstyle{definition}

\theoremstyle{definition}

\theoremstyle{definition}

\theoremstyle{definition}

\theoremstyle{definition}

\numberwithin{equation}{subsection}

\begin{document}
\title{Efficient search for superspecial hyperelliptic curves of genus four with automorphism group containing $\mathbb{Z}_6$}
%
%


\author{Momonari Kudo}
\address{Department of Mathematical Informatics, The University of Tokyo, Japan.}
\email{kudo@mist.i.u-tokyo.ac.jp}

\author{Tasuku Nakagawa}
\address{Department of Mathematical Informatics, The University of Tokyo, Japan.}
\email{nakagawa-tasuku705@g.ecc.u-tokyo.ac.jp}

\author{Tsuyoshi Takagi}
\address{Department of Mathematical Informatics, The University of Tokyo, Japan.}
\email{takagi@mist.i.u-tokyo.ac.jp}


\date{May 2022}
%
%
\maketitle              
%
%
%
%

\input{section1.tex}

\input{section2.tex}

\input{section3.tex}

\input{section4.tex}


%
%
%
%

\renewcommand{\baselinestretch}{1}

\input{section5.tex}
\end{document}

%% file: section1.tex
\begin{abstract}
In arithmetic and algebraic geometry, superspecial (s.sp.\ for short) curves are one of the most important objects to be studied, with applications to cryptography and coding theory.
If $g \geq 4$, it is not even known whether there exists such a curve of genus $g$ in general characteristic $p > 0$, and in the case of $g=4$, several computational approaches to search for those curves have been proposed.
In the genus-$4$ hyperelliptic case, Kudo-Harashita proposed a generic algorithm to enumerate all s.sp.\ curves, and recently Ohashi-Kudo-Harashita presented an algorithm specific to the case where automorphism group contains the Klein 4-group.
In this paper, we propose an algorithm with complexity $\tilde{O}(p^4)$ in theory but $\tilde{O}(p^3)$ in practice to enumerate s.sp.\ hyperelliptic curves of genus 4 with automorphism group containing the cyclic group of order $6$.
By executing the algorithm over Magma, we enumerate those curves for $p$ up to $1000$.
We also succeeded in finding a s.sp.\ hyperelliptic curve of genus $4$ in every $p$ with $p \equiv 2 \pmod{3}$.
As a theoretical result, we classify hyperelliptic curves of genus $4$ in terms of automorphism groups in the appendix.
\end{abstract}

\section{Introduction}\label{sec:intro}
Throughout, all the complexities are measured by the number of arithmetic operations in $\mathbb{F}_{p^2}$ for a prime $p$, unless otherwise noted.
Soft-O notation omits logarithmic factors.
A curve means a non-singular projective variety of dimension one.
Let $K$ be a field of characteristic $p>0$, and $\overline{K}$ its algebraic closure.
A curve $C$ of genus $g$ over $K$ is said to be {\it superspecial} ({\it s.sp}.\ for short) if its Jacobian variety is isomorphic to a product of supersingular elliptic curves.
S.sp.\ curves are of course important objects in theory, but also in practical applications such as cryptography using algebraic curves, see e.g., \cite{CDS}, where s.sp.\ genus-$2$ curves are used to construct Hash function.

Given a pair $(g,p)$, only finite s.sp.\ curves of genus $g$ over $\overline{\mathbb{F}_p}$ exist, and the problem of finding or enumerating them is known to be classically important.
For the field of definition, the most important case is $\mathbb{F}_{p^2}$, since any s.sp.\ curve over $K$ is $\overline{K}$-isomorphic to one over $\mathbb{F}_{p^2}$, see the proof of \cite[Theorem 1.1]{Ekedahl}.
For $g \leq 3$, the problem is solved for all $p >0$, based on the theory of principally polarized abelian varieties.
Specifically, for $g=1$ (resp.\ $2$ and $3$), Deuring~\cite{Deuring} (resp.\  Ibukiyama-Katsura-Oort \cite[Theorem 2.10]{IKO}) showed that the number of $\overline{\mathbb{F}_{p}}$-isomorphism classes of s.sp.\ curves is determined by computing the class numbers of a quaternion algebra (resp.\ quaternion hermitian lattices).
These class numbers were computed in \cite{Eichler} (resp.\ \cite{HI}, \cite{H}) for $g=1$ (resp.\ $2$, $3$). 

On the other hand, the problem for $g \geq 4$ has not been solved in all primes, but in recent years, Kudo-Harashita developed several algorithms to count genus-$4$ or $5$ s.sp.\ curves~\cite{KH17}, \cite{KH18}, \cite{KH20}.
In particular, an algorithm for enumerating s.sp.\ {\it hyperelliptic} curves of genus $4$ was proposed in \cite{KH18} and \cite{KH18b}, but is practical only for small $p$ (in fact $p \leq 23$), due to the cost of solving multivariate systems (cf.\ Section \ref{subsec:KHH} below).
Recently, Ohashi-Kudo-Harashita~\cite{OKH22} (resp.\ Kudo-Harashita-Howe~\cite{KHH}) presented an algorithm for enumerating s.sp.\ hyperelliptic (resp.\ non-hyperelliptic) curves of genus $4$ with automorphism group containing a subgroup isomorphic to the Klein $4$-group $V_4 = \mathbb{Z}_2 \times \mathbb{Z}_2$, with complexity $\tilde{O}(p^3)$ (resp.\ $\tilde{O}(p^4)$), where $\mathbb{Z}_n$ denotes the cyclic group of order $n$.
They also succeeded in enumerating such s.sp.\ curves for every prime $p$ up to $200$.

This paper proposes a more efficient algorithm than \cite{KH18} to produce s.sp.\ hyperelliptic curves of genus $4$,
which is practical for $p$ extremely larger than some number mentioned in \cite{KH18}.
For this, we focus on a family of hyperelliptic curves given by $H_{a,b} : y^2 = f_{a,b} (x):= x^{10} + x^7 + a x^4 + b x$, where $a, b \in \mathbb{F}_{p^2}$.
This kind of a curve appears as a s.sp.\ curve over $\mathbb{F}_{17^2}$ enumerated in \cite{KH18} (see also Table \ref{table:ssp} in Section \ref{subsec:KHH} below), and it tends to be s.sp.\ from our preliminary computation; by exhaustive search for $(a,b)$, we confirmed that there exists (resp.\ does not exist) $(a, b)$ such that $H_{a,b}$ is s.sp.\ for any $17 \leq p < 100$ with $p \equiv 2 \bmod{3}$ (resp.\ $p \equiv 1 \bmod{3}$).
We also note that the reduced (resp.\ full) automorphism group of $H_{a,b}$ contains a subgroup isomorphic to $\mathbb{Z}_3$ (resp.\ $\mathbb{Z}_6$), see Theorem \ref{thm:app} for a complete classification of reduced and full automorphism groups of hyperelliptic curves of genus $4$.

Here, main results of this paper are summarized in Theorems \ref{thm:main11} and \ref{thm:main22} below.

\begin{theor}\label{thm:main11}
There exists an algorithm (Main Algorithm in Theorem \ref{thm:main1}) with complexity $\tilde{O}(p^4)$ to enumerate the $\overline{\mathbb{F}_p}$-isomorphism classes of all s.sp. $H_{a,b}$'s with $a,b \in \mathbb{F}_{p^2}$.
Assuming the gcd of resultants appearing in the algorithm has degree $O(p)$, the complexity becomes $\tilde{O}(p^3)$.
\end{theor}

While the outline of our algorithm is same as that of our previous algorithm in \cite{KH18}, we shall develop various computational techniques specific to our family $H_{a,b}$.
Specifically, we first prove in Lemma \ref{lem:HabCM} that the Cartier-Manin matrix $M_{a,b}$ of $H_{a,b}$ with {\it parameters} $a$ and $b$ can be computed very efficiently, in $O(p^3)$, only with linear algebra.
We then solve the equation $M_{a,b} = 0$ in $\tilde{O}(p^4)$ with bivariate resultants, where the complexity becomes $\tilde{O}(p^3)$ assuming the gcd of computed resultants has degree $O(p)$;
we see from our computational results that this assumption is practical.
To make isomorphism classification of s.sp.\ $H_{a,b}$'s obtained as above efficient, we furthermore present some criteria based on our theoretical results provided in Appendix \ref{sec:app} on automorphism groups of genus-$4$ hyperelliptic curves.
For instance, it will be proved in Lemma \ref{lem:HabIsom} that two curves $H_{a,b}$ and $H_{a',b'}$ with reduced automorphism groups $\mathbb{Z}_3$ or $\mathbb{Z}_9$ are isomorphic, then $(a,b) = (a',b')$.
These criteria reduce the cost of isomorphism classification from $O(p^4)$ to $O(p^2)$ (in practice $O(p)$).

By implementing and executing our algorithm on Magma~\cite{Magma}, we succeeded in enumerating s.sp.\ $H_{a,b}$'s with $a,b \in \mathbb{F}_{p^2}$ up to isomorphisms over $\overline{\mathbb{F}_p}$ for every prime $p$ between $17$ and $1000$.
More precisely, we obtain the following computational results:

\begin{theor}\label{thm:main22}
For every prime $p$ with $17 \leq p < 1000$, the number of $\overline{\mathbb{F}_{p}}$-isomorphism classes of s.sp.\ $H_{a,b}$'s with $a,b \in \mathbb{F}_{p^2}$ are summarized in Tables \ref{table:1} -- \ref{table:4} below.
In particular, for each $17 \leq p < 1000$ with $p \equiv 2 \bmod{3}$ (resp.\ $p \equiv 1 \bmod{3}$), there exists (resp.\ does not exist) $(a,b) \in \mathbb{F}_{p^2}^2$ such that $H_{a,b}$ is a s.sp.\ hyperelliptic curve.
\end{theor}
The upper bound on $p$ in Theorem \ref{thm:main22} is much larger than those of \cite{KH18} and \cite{KH18b}, and it can be increased easily;
for instance, on a PC with macOS Monterey 12.0.1, at 2.6 GHz CPU 6 Core (Intel Core i7) and 16GB memory, it took 13,226 seconds (about 3.7 hours) in total for computing the $\overline{\mathbb{F}_p}$-isomorphicm classes of s.sp.\ $H_{a,b}$'s with $a,b \in \mathbb{F}_{p^2}$ for all $17 \leq p < 1000$, and the execution time for $p=997$ was only 196 seconds.

As a theoretical result, we shall also give an explicit classification of hyperelliptic curves of genus $4$ completely, in terms of their automorphism groups, in Theorem \ref{thm:app} of Appendix \ref{sec:app}.
With this classification, we see that our family $H_{a,b}$ is included in the types {\bf 3}, {\bf 7} and {\bf 9}, while Ohashi-Kudo-Harashita's recent work~\cite{OKH22} treats the types {\bf 2-1}, {\bf 4-1}, {\bf 5}, {\bf 6}, {\bf 8} and {\bf 10}.

The rest of this paper is organized as follows.
Section 2 is devoted to preliminaries, where we review general facts on hyperelliptic curves, Cartier-Manin matrices, and enumeration results in \cite{KH18}, \cite{KH18b} on s.sp.\ hyperelliptic curves.
Section 3 provides the main results.
Section 4 is conclusion.

%% file: section2.tex
\section{Preliminaries}\label{sec:pre}

This section reviews general facts on hyperelliptic curves, and recalls the definition of Cartier-Manin matrices and the superspeciality of curves.
In particular, we will describe a method to compute the Cartier-Manin matrix of a hyperelliptic curve.
Some facts on the Cartier operator of a cyclic cover will be also reviewed.
We also review Kudo-Harashita's algorithm~\cite{KH18}, \cite{KH18b} to enumerate superspecial hyperelliptic curves, and their enumeration results.

\subsection{Hyperelliptic curves and their isomorphisms}

Let $K$ be a field of characteristic $p$ with $p \neq 2$, and $k=\overline{K}$ its algebraic closure.
For a curve $C$ of genus $g \geq 2$ over $K$, let $\mathrm{Aut}_K(C)$ denote the automorphism group of $C$ over $K$, and $\mathrm{Aut}_{k}(C)$ is denoted simply by $\mathrm{Aut}(C)$.
It is well-known that $\mathrm{Aut}(C)$ is finite for an arbitrary $C$, and has size $\leq 16 g^4$ unless $C$ is a Hermitian curve~\cite{Sti}.
If the characteristic of $K$ exceeds $g+1$, we have a quite more strong bound $\mathrm{Aut}(C) \leq 84 (g-1)$, see \cite{Roq}.
A {\it hyperelliptic curve} $H$ over $K$ is a curve $H$ over $K$ admitting a degree-$2$ morphism over $K$ from $H$ to the projective line $\mathbb{P}_k^1$.
Let $\iota$ be the {\it hyperelliptic involution} of $H$, that is, the unique involution over $k$ on $C$ such that the quotient curve $C / \langle \iota \rangle$ is rational.
We call the quotient group $\overline{\rm Aut}(H):=\mathrm{Aut}(H) / \langle \iota \rangle$ the {\it reduced automorphism group} of $H$, while $\mathrm{Aut}(H)$ and $\mathrm{Aut}_K(H)$ are often called {\it full} automorphism groups.

A typical way to represent a hyperelliptic curve $H$ explicitly is realizing it as the desingularization of the projective closure of an affine plane curve $y^2 = f(x)$, where $f(x) \in k[x]$ is a separable polynomial of degree $2g +1$ or $2g+2$.
In this situation, we simply write $H : y^2 = f(x)$ and call the equation $y^2 = f(x)$ a (hyperelliptic) equation of $H$.
We can also write down an equation of $H$ in terms of a field $K$ of definition for $H$:

\begin{lem}[{\cite[Lemma 2]{KH18}}]\label{ReductionHyper}
Let $H$ be a hyperelliptic curve of genus $g$ over $K$.
Assume that $p$ and $2g+2$ are coprime, and let $\epsilon \in K^\times \smallsetminus (K^\times)^2$.
Then $H$ is birational to the projective closure of
\begin{equation}\label{eq:hyp}
c y^2 = f(x) =  x^{2g+2} + b x^{2g} + a_{2g-1}x^{2g-1} + \cdots + a_1 x + a_0,
\end{equation}
where $a_i \in K$ for $0 \leq i \leq 2g-1$, and where $b= 0, 1,\epsilon$ and $c=1,\epsilon$.
\end{lem}

The following lemma gives a criterion to test whether two hyperelliptic curves over $K$ are $K$-isomorphic to each other, or not:

\begin{lem}[{\cite[Section 1.2]{LR}} or {\cite[Lemma 1]{KH18}}]\label{lem:isom}
Let $H_i: c_i y^2=f_i(x)$ be hyperelliptic curves of genus $g$ over $K$ for $i=1$ and $2$, where $c_i y^2 = f_i(x)$ is of the form \eqref{eq:hyp}.
For any $K$-isomorphism $\sigma \colon H_1 \to H_2$, there exists $(P,\lambda) \in \mathrm{GL}_2(K) \times K^{\times}$ with
\[
P=
\begin{pmatrix}
\alpha & \beta \\
\gamma & \delta
\end{pmatrix} 
\]
such that
\[
\sigma (x,y) = \left( \frac{\alpha x + \beta}{\gamma x + \delta}, \frac{\lambda y}{(\gamma x + \delta)^{g+1}} \right)
\]
for all $(x,y)$ on $H_1$.
The representation of $\sigma$ is unique up to the equivalence $(P,\lambda) \sim (\mu P, \mu^{g+1} \lambda)$ for $\mu \in K^{\times}$.
\end{lem}

\begin{rem}\label{rem:iso}
Considering a hyperelliptic curve $H : y^2 = f(x)$ over an algebraically closed field $k$, we may assume that the coefficients of the highest and lowest degree terms in $f(x)$ are $1$.
Indeed, writing $f(x) = \sum_{i=s}^d a_i x^i$ with $d= 2g+1, 2g+2$, $s=0,1$, $a_d \neq 0$ and $a_s \neq 0$, an isomorphism $(x,y) \mapsto \left( \frac{\alpha x}{\delta}, \frac{y}{\delta^{g+1}} \right)$ transforms $y^2 = f(x)$ into $y^2 = \sum_{i=s}^d a_i \alpha^i \delta^{2g+2 - i} x^i $ choosing $\alpha$ and $\delta$ so that $a_{d} \alpha^{d} \delta^{2g+2-d}  = 1$ and $a_s \alpha^s \delta^{2g+2-s} =1$, as desired.
\end{rem}

\subsection{Cartier-Manin matrices and superspeciality}

In this section, we review how to compute the Cartier-Manin matrix of a hyperelliptic curve.

We start with recalling the definition of the Cartier operator and the Cartier-Manin matrix for a general curve which admits an affine plane model.
Assume for simplicity that $C$ is birational to an affine plane (possibly singular) curve $F(x,y) = 0$ in $\mathbb{A}^2$ with coordinate ring $R:=k[x,y]/ \langle F \rangle $, where $F$ is an irreducible polynomial over $k$ in $x$ and $y$.
We may identify the function field $k(C)$ and the field of fractions $K:=k(x,y)$ for $R$.
Under this identification, every regular differential form $\omega \in H^0 (C, \Omega_C^1)$ is uniquely written as $\omega = d \phi + \eta^p x^{p-1} d x$ for $\phi, \eta \in k(C)$.
Here we define a map
\[
\mathscr{C} : H^0 (C, \Omega_C^1) \to H^0 (C, \Omega_C^1)
\]
by $\mathscr{C} (\omega) := \eta dx$, and call it the (modified) {\it Cartier operator} on $H^0 (C, \Omega_C^1) $.
Moreover, the matrix representing $\mathscr{C}$ with respect to a basis $\mathcal{A}$ for the $g$-dimensional space $H^0(C, \Omega_C^1)$ is called the {\it Cartier-Manin matrix} of $C$.
We here also recall Nygaard's criterion for superspeciality in terms of the Cartier operator:

\begin{thm}[{\cite[Theorem 4.1]{Nygaard}}]\label{thm:Nygaard}
With notation as above, the Jacobian variety $J(C)$ of a curve $C$ is isomorphic to a product of supersingular elliptic curves if and only if $\mathscr{C}$ vanishes.
\end{thm}

In the case where $C$ is hyperelliptic, we have a well-known explicit formula (Lemma \ref{lem:ssp} below) by Yui~\cite{Yui} to compute the Cartier-Manin matrix of $C$, and so recall it here.
As in the previous subsection, assume that $C$ is a hyperelliptic curve of genus $g$ over $k$ defined by $y^2 = f(x)$, where $f(x)$ is a polynomial in $k[x]$ of degree $2g+1$ or $2g+2$ with no multiple root. 
First, it is well-knwon that a basis of $H^0 (C,\Omega_C^1)$ is given by
\[
\mathcal{A} = \left\{\omega_j :=\frac{x^{j-1}}{y}dx : 1\le j \le g\right\}.
\]
Writing $f(x)^{(p-1)/2} = \sum_k c_k x^k$ for $c_k \in k$, it follows from $y^{p-1} = f(x)^{(p-1)/2}$ in $k(C)$ that
\begin{eqnarray*}
\omega_j &=& y^{-p}f(x)^{(p-1)/2}x^{j-1}dx\\
&=&
d\left(y^{-p}\sum_{\substack{k\\ j+k \not\equiv 0 \pmod p}}\frac{c_k}{j+k} x^{j+k}\right) + \sum_{i \geq 1} c_{ip-j} \frac{x^{(i-1)p}}{y^p} x^{p-1} dx.
\end{eqnarray*}
Therefore
\[
{\mathscr C} (\omega_j ) = \sum_{i=1}^g c_{ip-j}^{1/p} \omega_{i}
\]
by the definition of the Cartier operator described above, and hence we have the following lemma:

\begin{lem}[{\cite[Section 2]{Yui}}]\label{lem:ssp}
With notation as above, the Cartier-Manin matrix of $C$ is the $g \times g$ matrix whose $(i,j)$-entry is the coefficient $c_{ip-j}$ of $x^{pi -j}$ in $f^{(p-1)/2}$ for $1 \leq  i, j \leq g$.
Hence, by Theorem \ref{thm:Nygaard}, $C$ is superspecial if and only if the coefficients of $x^{p i - j}$ in $f^{(p-1)/2}$ are equal to $0$ for all pairs of integers $1 \leq i , j \leq g$.
\end{lem}

\begin{ex}\label{ex:CM}
Consider the hyperelliptic curves $H_1 : y^2 = x^{2g+2} + x$, $H_2 : y^2 = x^{2g+1} + x$ and $H_3 : y^2 = x^{2g+2} + 1$ over $\mathbb{F}_{p}$.
The reduced automorphism group of the first curve has a subgroup isomorphic to $\mathbb{Z}_{2g+1}$, and those of the second and third ones are isomorphic to $D_{2g}$ and $D_{2g+2}$ respectively, by Lemma \ref{lem:Dg}. 
Since $(x^{2g+2} + x)^{\frac{p-1}{2}} = \sum_{k=0}^{\frac{p-1}{2}} \binom{\frac{p-1}{2}}{k} x^{(2g+1)k + \frac{p-1}{2}}$, the Cartier-Manin matrix of $H_1$ is zero if $i p - j \not\equiv \frac{p-1}{2} \pmod{2g+1}$ for any $1 \leq i,j \leq g$.
For instance, if $p \equiv -1 \pmod{2g+1}$, then $H_1$ is superspecial.
Similarly, $H_2$ (resp.\ $H_3$) is superspecial if $p \equiv -1 \pmod{2g}$ (resp.\ $p \equiv -1 \pmod{2g+2}$).

More strongly, it is proved in \cite{Taf12} that $H_1$ (resp.\ $H_2$) is $\mathbb{F}_{p^2}$-maximal if and only if $p \equiv -1 \pmod{2g+1}$ (resp.\ $p \equiv -1, 2g+1 \pmod{4g}$).

\end{ex}

Lemma \ref{lem:ssp} reduces the computation of the Cartier-Manin matrix of $C : y^2 = f(x)$ into that of $g^2$ (particular) coefficients in the power $f^{(p-1)/2}$.
In the case where $k$ is a finite field (with no parameter), several efficient algorithms to compute the coefficients have been proposed by Bostan-Gaudry-Schost~\cite{BGS}, Komoto-Kozaki-Matsuo~\cite{KKM} and Harvey-Sutherland~\cite{HS}.
These algorithms commonly use a linear recurrence by Flajolet-Salvy~\cite{FS} (described also in e.g., \cite[Section 4]{BGS}) which is used to a general method to compute the power of a given univariate polynomial.

We here recall the recurrence since it will be requred to analyze the complexity of our main algorithm.
Let $h(x)$ be a univariate polynomial of degree $d$, and let $\frac{dh}{dx}$ denote its derivative with respect to $x$.
We also denote by $h_i$ its $x^i$-coefficient for each $0 \leq i \leq d$, say $h(x) = \sum_{i=0}^d h_i x^i$.
Let $n$ be a positive integer, and we consider to compute the power $h^n$.
For each $k$ with $0 \leq k \leq (n+1) d$, it follows from $h^{n+1} = h h^n$ that
\begin{equation}\label{eq:power}
(h^{n+1})_k = \sum_{j=0}^d h_j (h^n)_{k-j},
\end{equation}
where $(h^n)_{-d}=(h^n)_{-d+1}=\cdots = (h^n)_{-1} = (h^n)_{nd+1} = (h^n)_{nd+2} = \cdots = (h^n)_{(n+1)d} = 0$.
On the other hand, it also follows from $\frac{d}{dx} (h^{n+1}) = (n+1) \frac{dh}{dx} \cdot h^n$ that
\begin{equation}\label{eq:derivative}
k (h^{n+1})_k = (n+1) \sum_{i=0}^{d-1} \left( \frac{d h}{dx} \right)_{i} (h^n)_{k-1-i} = (n+1) \sum_{j=1}^{d} j h_j (h^n)_{k-j}  
\end{equation}
by comparing the $x^{k-1}$-coefficient.
Multiplying \eqref{eq:power} by $k$ and subtracting \eqref{eq:derivative}, we have a linear recurrence
\begin{equation}\label{eq:relation}
\sum_{j=0}^d (n j - k + j ) h_j (h^n)_{k-j} = 0,
\end{equation}
equivalently,
\begin{equation}\label{eq:relation1}
k h_0 (h^n)_k = \sum_{j=1}^d (n j - k + j ) h_j (h^n)_{k-j}.
\end{equation}
Thus, if both $h_0$ and $k$ are not equal to zero in the coefficient ring of $h$, the coefficient $(h^n)_k$ is a linear combination of $d$ {\it lower}-degree coefficients $(h^n)_{k-1}, \ldots , (h^n)_{k-d}$, say
\begin{equation}\label{eq:relation1-2}
    (h^n)_{k} = \sum_{j=1}^d \frac{n j - k + j}{h_0 k} h_j (h^n)_{k-j} ,
\end{equation}
and hence
\begin{equation}\label{eq:mat}
U_k := 
\begin{pmatrix}
(h^n)_{k-d+1} \\
(h^n)_{k-d+2} \\
\vdots \\
(h^n)_{k}
\end{pmatrix}
= A(k) U_{k-1} = A(k) A(k-1) \cdots A(1) U_0 ,
\end{equation}
where we set
\begin{equation}\label{eq:mat2}
A(k) := 
\begin{pmatrix}
0 & 1 &  \cdots &  0  \\
0 & 0 & \cdots & 0 \\
\vdots & \vdots & \ddots & 0\\
0 & 0 & \cdots & 1 \\
r_d(k) & r_{d-1}(k) & \cdots & r_1(k)
\end{pmatrix}
\end{equation}
with
\[
r_j (k) = \frac{(n j - k + j) h_j}{h_0 k}
\]
for $1 \leq j \leq d$.
Therefore, the coefficients $(h^n)_k$ for all $1 \leq k \leq m$ can be obtained by recursively computing the vectors $U_{k}$ with \eqref{eq:mat}, starting from $k=1$ up to $m$, for the initial value $U_0$ with $(h^n)_{-d+1}=(h^n)_{-d+2}=\cdots = (h^n)_{-1} = 0$ and $(h^n)_0 = (h_0)^{n}$.
Note that each $U_k$ is computed unless $k \neq 0$ in the coefficient ring $R$ of $h$.
Therefore, this recursive computation is always valid for the characteristic $0$ case, but it may not be applied directly in the positive characteristic case since the value of $k$ can be zero in $R$:
One solution is lifting to characteristic $0$ such as $\mathbb{F}_{p}$ to $\mathbb{Q}_p$.

This method (with lifting to characteristic zero if necessary) can be applied to computing the Cartier-Manin matrix of the hyperelliptic curve $C : y^2=f(x)$ as follows:
If $f_0 \neq 0$, simply put $n=(p-1)/2$ and $h = f$, and then the $(ip-j)$-th coefficients $(f^n)_{ip-j}$ with $1 \leq i,j \leq g$ are computed as entries of $U_k$ for $1 \leq k \leq gp-1$.
Otherwise, putting $f = x h(x)$, it follows from $f^n = x^n h^n$ that $(f^n)_{ip-j}$ is equal to the $\left( \frac{(2i-1)p - (2j-1)}{2} \right)$-th coefficient of $h^n$.
Hence, it suffices to compute $U_k$ for $1 \leq k \leq \frac{(2g-1)p-1}{2}$.

Bostan-Gaudry-Schost~\cite{BGS} constructed an efficient algorithm to compute the Cartier-Manin matrix of $C$ over $\mathbb{F}_{p^r}$, by using the above recurrences with lifting to the unramified extension of $\mathbb{Q}_p$ of degree $r$.

\begin{rem}
Unlike \eqref{eq:relation1-2}, each coefficient $(h^n)_k$ can be represented as a linear combination of $d$ {\it higher}-degree coefficients $(h^n)_{k+1}, \ldots , (h^n)_{k+d}$.
More precisely, it follows from \eqref{eq:relation} that
\begin{equation}\label{eq:relation2}
(nd - k + d) h_d (h^n)_{k-d} = - \sum_{j=0}^{d-1} (n j - k + j) h_j (h^n)_{k-j}.
\end{equation}
Replacing $k$ by $k+d$, one has
\begin{equation}\label{eq:relation2-2}
(nd - k) h_d (h^n)_{k} = - \sum_{j=0}^{d-1} (n j - k -(d- j)) h_j (h^n)_{k+d-j}
\end{equation}
for $-d \leq k \leq nd$.
Putting $i=d-j$, one also has
\begin{equation}\label{eq:relation3}
(nd-k) h_d (h^n)_{k} = - \sum_{i=1}^{d} (n (d-i) - k -i) h_{d-i} (h^n)_{k+i}
\end{equation}
for $-d \leq k \leq nd$.
Therefore, we can compute $(h^n)_k$ from $(h^n)_{k+1}, \ldots , (h^n)_{k+d}$ as in \eqref{eq:relation1-2} unless $n d - k \neq 0$, and can consider a recursive computation as in \eqref{eq:mat} starting from $(h^n)_{nd} = (h_d)^n$.
\end{rem}

\subsection{Cyclic covers of the projective line and their Cartier operators}\label{subsec:cyclic}

In this subsection, we briefly review Elkin's results~\cite{Elkin} on the rank of the Cartier operator for a cyclic cover, since they will be applied in Section \ref{subsec:CMHab} below to our family of hyperelliptic curves $H_{a,b}$ having order-$6$ automorphisms.

Let $C$ be a (non-singular) curve of genus $g$ over an algebraically closed field $k$ of characteristic $p > 2$.
Assume that there exists a ramified Galois cover $\pi : C \to \mathbb{P}^1$ of degree $n$, and let $r$ be the number of ramification points in $\mathbb{P}^1$ of $\pi$, where $n$ is coprime to $p$.
In this case, for each ramification point $P_i$ in $\mathbb{P}^1$of $\pi$, its preimage $\pi^{-1} ( \{ P_i \})$ consists of $n/v_i$ branched points in $C$ with the same ramification index $v_i$ with $2 \leq v_i \leq n$ dividing $n$.
By Hurwitz’s formula, we have
\begin{equation}\label{eq:Hurwitz}
   2g -2 + 2 n = \sum_{i=1}^r \frac{n}{v_i} (v_i-1),
\end{equation}
and $C$ is birational to the homogenization of
\begin{equation}\label{eq:spe}
   y^n = (x-a_1)^{n_1} \cdots (x-a_r)^{n_r}
\end{equation}
for some mutually distinct elements $a_1,\ldots , a_r \in k$, and some integers $n_i$ with $1 \leq n_i < n$ and $v_i = n / \mathrm{gcd}(n,n_i)$, and $\mathrm{gcd}(n,n_1, \ldots , n_r) = 1$ such that $\sum_{i=1}^r n_i \equiv 0 \pmod{n}$. 
In this case, we say that $C \to \mathbb{P}^1$ (or $C$ simply) is a cyclic cover of type $(n;n_1, \ldots , n_r)$.

The group of $n$-th roots of unity in $k$ acts on $H^0 (C, \Omega_C^{1})$.
More precisely, let $\zeta_n$ be a primitive $n$-th root of unity in $k$, and $\delta$ the automorphism on $C$ defined by $(x,y) \mapsto (x, \zeta_n^{-1} y)$ (namely $\delta$ is a generator of $\mathrm{Aut}(C)$).
Then $\delta$ induces an automorphism $\delta^{\ast}$ of the linear space $H^0 (C,\Omega_C^1)$.
Denoting by $D_i$ the $\zeta_n^i$-eigenspace of $\delta^{\ast}$, we decompose
\begin{equation}\label{eq:dec}
H^0 (C, \Omega_{C}^1) = \bigoplus_{i=0}^{n-1} D_i ,
\end{equation}
where the dimension $d_i$ of $D_i$ is given as
\begin{equation}\label{eq:dim_di}
d_i = \left( \sum_{j=1}^r \frac{(i n_j \bmod{n})}{n} \right) - 1.
\end{equation}
By the $p^{-1}$-linearity of the Cartier operator $\mathscr{C}$, we have $\mathscr{C}(\lambda^{pi} \omega) = \lambda^i \mathscr{C}(\omega)$ for any differential $\omega$, and thus
\begin{equation}\label{eq:inc}
\mathscr{C}(D_{p i \bmod{n}} ) \subset D_i
\end{equation}
for every $0 \leq i \leq n-1$.

By the decomposition \eqref{eq:dec}, we have
\begin{equation}\label{eq:sum}
\mathrm{rank} (\mathscr{C}) =  \sum_{i=0}^{n-1} \mathrm{dim}(\mathscr{C}(D_i)),
\end{equation}
each term of which satisfies the following:

\begin{thm}[{\cite[Theorem 1.1]{Elkin}}]
With notation as above, we have the following inequalities:
\[
\mathrm{min}(2 \lfloor d_i/p \rfloor, d_{\sigma(i)}) \leq \mathrm{dim}(\mathscr{C}(D_i)) \leq \mathrm{min} (d_i,d_{\sigma(i)}),
\]
where $\lfloor \cdot \rfloor$ denotes the floor function, and where $\sigma$ is the inverse of a permutation map on the set $\{ 0, 1, \ldots , n-1 \}$ given by $i \mapsto p i \bmod{n}$.
\end{thm}

Therefore, it follows from $\eqref{eq:sum}$ that we obtain:
\[
\sum_{i=0}^{n-1} \mathrm{min}(2 \lfloor d_i/p \rfloor, d_{\sigma(i)}) \leq \mathrm{rank} (\mathscr{C}) \leq \sum_{i=0}^{n-1} \mathrm{min} (d_i,d_{\sigma(i)}),
\]
see {\cite[Corollary 4.3]{Elkin}}.


\subsection{Kudo-Harashita's enumeration of superspecial hyperelliptic curves}\label{subsec:KHH}

Based on Lemmas \ref{ReductionHyper}, \ref{lem:isom} and \ref{lem:ssp}, an algorithm for enumerating superspecial hyperelliptic curves over $K=\mathbb{F}_q$ with $q=p$ or $p^2$ was proposed by Kudo-Harashita~\cite{KH18}, \cite{KH18b}, and it consists of the following three steps:
\begin{enumerate}
    \item Regarding unknown coefficients in \eqref{eq:hyp} as variables, compute the Cartier-Manin matrix of $H$ defined by \eqref{eq:hyp}.
    \item Fixed constants $b$ and $c$ in \eqref{eq:hyp}, compute the roots over $\mathbb{F}_q$ of the multivariate system ``the Cartier-Manin matrix is zero'' with $2g$ variables $a_{2g-1}, \ldots , a_0$ by the hybrid approach~\cite{BFP} mixing Gr\"{o}bner basis computation and exhaustive search.
    \item Classify the collected curves corresponding to the roots of the system into isomorphism classes, by Lemma \ref{lem:isom}.
\end{enumerate}
Kudo-Harashita implemented the algorithm on Magma~\cite{Magma}, and executed it for the case $g=4$ with $q=11^2, 13^2, 17^2, 19^2, 23$.
According to \cite[Section 3.1]{KH18b}, they succeeded in finishing required computation within a day in total.
The main results in \cite{KH18} and \cite{KH18b} are the following:

\begin{thm}[{\cite[Theorem 1]{KH18}}]\label{thm:KH18-1}
There is no superspecial hyperelliptic curve of genus $4$ in characteristic $p$ with $p \leq 13$.
\end{thm}

\begin{thm}[{\cite[Theorem 2]{KH18}}]\label{thm:KH18-2}
There exist precisely $5$ $($resp.\ $25)$ superspecial hyperelliptic curves of genus $4$ over $\mathbb{F}_{17}$ $($resp.\ $\mathbb{F}_{17^2})$ up to isomorphism over $\mathbb{F}_{17}$ $($resp.\ $\mathbb{F}_{17^2})$.
Moreover, there exist precisely $2$ superspecial hyperelliptic curves of genus $4$ over $\overline{\mathbb{F}_{17}}$ up to isomorphism.
\end{thm}

\begin{thm}[{\cite[Theorem 3]{KH18}, \cite[Theorem 3]{KH18b}}]\label{thm:KH18-3}
There exist precisely $12$ $($resp.\ $18)$ superspecial hyperelliptic curves of genus $4$ over $\mathbb{F}_{19}$ $($resp.\ $\mathbb{F}_{19^2})$ up to isomorphism over $\mathbb{F}_{19}$ $($resp.\ $\mathbb{F}_{19^2})$.
Moreover, there exist precisely $2$ superspecial hyperelliptic curves of genus $4$ over $\overline{\mathbb{F}_{19}}$ up to isomorphism.
\end{thm}

\begin{thm}[{\cite[Theorem 4]{KH18b}}]\label{thm:KH18-4}
There exist precisely $14$ superspecial hyperelliptic curves of genus $4$ over $\mathbb{F}_{23}$ up to isomorphism over $\mathbb{F}_{23}$.
Moreover, there exist precisely $4$ superspecial hyperelliptic curves of genus $4$ over $\mathbb{F}_{23}$ up to isomorphism over $\overline{\mathbb{F}_{23}}$.
\end{thm}

As examples, the $\overline{\mathbb{F}_{p}}$-isomorphism classes of superspecial hyperelliptic curves of genus $4$ defined over the prime field $\mathbb{F}_{p}$ are summarized in Table \ref{table:ssp} below.
Note that the reduced automorphism group of {\it every} superspecial curve in Table \ref{table:ssp} is non-trivial.

\renewcommand{\arraystretch}{1.3}
\begin{table}[H]
\centering{
\caption{The $\overline{\mathbb{F}_p}$-isomorphism classes of all superspecial hyperelliptic curves $H$ of genus $4$ over the prime field $\mathbb{F}_{p}$ for $p=17$, $19$ and $23$.
For $n \geq 2$, we denote respectively by $\mathbb{Z}_n$, $D_n$, $A_n$ and $V_4$ the cyclic group of order $n$, the dihedral group of order $2n$, the alternating group of order $n!/2$, and the Klein $4$-group $\mathbb{Z}_2 \times \mathbb{Z}_2$.}
\label{table:ssp}
\vspace{5pt}
\scalebox{0.8}{
\begin{tabular}{c||l|c|c|c} \hline
$p$ & Equation of $H$ representing an isomorphism class & $\overline{\mathrm{Aut}}(H)$ & $\mathrm{Aut}(H)$ & Case in Table \ref{table:aut} \\ \hline
17 & $y^2 = x^{10}+x$ & $\mathbb{Z}_9$ & $\mathbb{Z}_{18}$ & {\bf 9} \\ 
 & $y^2 = x^{10}+x^7 + 13 x^4 + 12x$ & $A_4$ & $\mathrm{SL}_2(\mathbb{F}_3)$ & {\bf 7} \\ \hline
19 & $y^2 = x^{10}+1$ & $D_{10}$  & $\mathbb{Z}_5 \rtimes D_4$ & {\bf 10} \\ 
 & $y^2 = x^{10} + x^7 + 4 x^6 + 15 x^5 + 6 x^4 + 8 x^3 + 5 x^2 + 12 x + 1$ & $V_4$ & $D_4$ & {\bf 4-1} \\ \hline
 23 & $y^2 = x^{10}+x^7 + 3 x^4 + 10 x$ & $\mathbb{Z}_3$ & $\mathbb{Z}_6$ & {\bf 3} \\ 
 & $y^2 = x^{10} + x^7 + 18 x^4 + 6x$ & $A_4$ & $\mathrm{SL}_2(\mathbb{F}_3)$ & {\bf 7} \\ 
 & $y^2 = x^{10} + x^7 + 5 x^6 + 3 x^5 + 21 x^4 + 3 x^3 + 9 x^2 + 4 x + 21$ & $V_4$ & $D_4$ & {\bf 4-1} \\
  & $y^2 = x^{10} + x^7 + 9 x^6 + 11 x^5 + 19 x^4 + 10 x^3 + 16 x^2 + 8 x + 21$ & $\mathbb{Z}_2 $ & $V_4$ & {\bf 2-1} \\ \hline
\end{tabular}
}
}
\end{table}
\renewcommand{\arraystretch}{1}

While Kudo-Harashita succeeded in enumerating superspecial hyperelliptic curves for concrete $p$, the complexity of their algorithm has not been investigated, due to the difficulty of estimating the cost of Gr\"{o}bner basis computation.
In fact, it might be exponential with respect to $p$, since the multivariate system to be solved in Step 2 has the maximal total-degree $(p-1)/2$.
Moreover, Step 3 might also be costly, due to the growth of the number of solutions found in Step 2.

To overcome the limitation of the enumeration in practical time, Ohashi-Kudo-Harashita~\cite{OKH22} recently proposed an efficient algorithm with complexity $\tilde{O}(p^3)$ for enumerating superspecial hyperelliptic curves of genus $4$, focusing on the space of those curves with extra involution (see also \cite{MK22} for the genus-$3$ case, and \cite{KHH} for the genus-$4$ non-hyperelliptic case).
Namely, the algorithm in \cite{OKH22} treats just the case {\bf 2-1} (and the cases {\bf 4-1}, {\bf 5}, {\bf 6}, {\bf 8}, {\bf 10}) of Table \ref{table:aut}, i.e., ${\rm Aut}(H) \supset V_4$ whereas this paper focuses on the cases {\bf 3}, {\bf 7} and {\bf 9}, i.e., ${\rm Aut}(H) \supset \mathbb{Z}_6$.



%% file: section3.tex
\section{Main results}

As in the previous sections, let $k$ be an algebraically closed field of characteristic $p$ with $p \geq  7$.
In this section, we shall present an algorithm to efficiently produce superspecial hyperelliptic curves of genus $4$, by focusing on the 
following parametric family:
\begin{equation}\label{eq:Hab}
H_{a,b} : y^2 = f_{a,b} (x):= x^{10} + x^7 + a x^4 + b x,
\end{equation}
where $a, b \in k$.
(The reason why we focus on this family is described in Section \ref{sec:intro}.)
The (resp.\ reduced) automorphism group of this curve contains a subgroup isomorphic to $\mathbb{Z}_6$ (resp.\ $\mathbb{Z}_3$), which is included in the cases {\bf 3}, {\bf 7} and {\bf 9} of Table \ref{table:aut} in Theorem \ref{thm:app}.
More precisely, denoting by $\zeta_3$ a primitive $3$rd root of unity in $k$, this curve has an order-$3$ automorphism $\sigma : (x,y) \mapsto (\zeta_3 x, \zeta_3^2 y)$ represented by $(A,\lambda)$ with $A = \mathrm{diag}(\zeta_3,1)$ and $\lambda = \zeta_3^2$.
Note that $\zeta_3 \in \mathbb{F}_{p^2}$ since $\zeta_3$ is a root of $x^2+x+1\in \mathbb{F}_p[x]$.
It is also straightforward that there exists a degree-$3$ map $H_{a,b} \to H_{a,b}/\langle \sigma \rangle$, where the quotient curve $H_{a,b}/\langle \sigma \rangle$ is a genus-one curve given by $Y^2 = X(X^3 + X^2 + a X + b)$.

\subsection{Cartier-Manin matrices of our curves}\label{subsec:CMHab}

Let $M_{a,b}$ denote the Cartier-Manin matrix of $H_{a,b}$.
In this subsection, we shall determine the form of $M_{a,b}$, and investigate its rank by using Elkin's method~\cite{Elkin} for cyclic coverings of $\mathbb{P}^1$.

\begin{lem}
With notation as above, the Cartier-Manin matrix $M_{a,b}$ of $H_{a,b}$ is given as follows:
\begin{enumerate}
    \item[(1)] If $p \equiv 1 \pmod{3}$, then
    \[
    M_{a,b}=
    \begin{pmatrix}
    c_{p-1} & 0 & 0 & c_{p-4} \\
    0 & c_{2p-2} & 0 & 0 \\
    0 & 0 & c_{3p-3} & 0 \\
    c_{4p-1} & 0 & 0 & c_{4p-4}
    \end{pmatrix}.
    \]
    \item[(2)] If $p \equiv 2 \pmod{3}$, then
    \[
    M_{a,b}=
    \begin{pmatrix}
    0 & 0 & c_{p-3} & 0 \\
    0 & c_{2p-2} & 0 & 0 \\
    c_{3p-1} & 0 & 0 & c_{3p-4} \\
    0 & 0 & c_{4p-3} & 0
    \end{pmatrix}.
    \]
\end{enumerate}
Here $c_{k}$ denotes the $x^k$-coefficient of $f_{a,b}^{(p-1)/2}$.
\end{lem}

\begin{proof}
Since
\[
f_{a,b}(x)^{\frac{p-1}{2}} = \sum_{k_1 + k_2 + k_3 + k_4 = \frac{p-1}{2}} \binom{\frac{p-1}{2}}{k_1,k_2,k_3,k_4} a^{k_3} b^{k_4} x^{9k_1 + 6k_2 + 3 k_3 + \frac{p-1}{2}},
\]
the coefficient of each $x^k$ in $f_{a,b}(x)^{\frac{p-1}{2}}$ is zero if $k \not\equiv \frac{p-1}{2} \pmod{3}$.
Computing $1 \leq i, j \leq g$ with $i p - j \equiv \frac{p-1}{2} \pmod{3}$ dividing the case into $p \equiv 1 \pmod{3}$ and $p \equiv 2 \pmod{3}$, we obtain the assertion by Lemma \ref{lem:ssp}. 
\end{proof}

Next, we shall determine the type of $H_{a,b}$ as a cyclic cover to the projective line, in order to apply Elkin's method~\cite{Elkin} for obtaining bounds on the rank of the Cartier-Manin matrix $M_{a,b}$.

\begin{lem}
With notation as above, we have the following:
\begin{enumerate}
    \item[(1)] $H_{a,b}$ is a cyclic cover of type $(6;1,2,3,3,3)$, i.e., it is birational to
    \[
    y^6 = (x- a_1 )^1 (x- a_2)^2 (x - a_3)^3 (x - a_4)^3 (x - a_5)^3 
    \]
    for some $a_i \in k$ with $1 \leq i \leq 5$.
    \item[(2)] The rank of the Cartier-Manin matrix of $H_{a,b}$ is equal to or smaller than $3$ if $p \equiv 2 \pmod{3}$.
\end{enumerate}
\end{lem}

\begin{proof}
To prove (1), we construct a cyclic cover $H_{a,b} \to  H_{a,b} / \langle \sigma \circ \iota \rangle \cong \mathbb{P}^1$ of degree $6$, and determine its ramification points.
We can write
\[
y^2 = f_{a,b}(x) = x (x^3 - A_1) (x^3 - A_2) (x^3 - A_3)
\]
for some $A_i \in k$ with $1 \leq i \leq 3$.
The curve $H_{a,b}$ has an automorphism $\sigma \circ \iota : (x,y) \mapsto (\zeta x , -\zeta^2 y)$ of order $6$, where $\iota$ is the hyperelliptic involution, and where $\zeta$ is a primitive $3$rd root of unity in $k$.

We here define a cyclic cover $\pi : H_{a,b} \to H_{a,b}/\langle \sigma \circ \iota \rangle \ ; \ (x : y : z) \mapsto (x^3 : x^2 y^2 : z)$ of degree $6$ as the composition of $H_{a,b} \to H_{a,b}/\langle \sigma \rangle$ given by $ (x:y:z) \mapsto (x^3:x y:z)$ and $H_{a,b}/\langle \sigma \rangle \to H_{a,b} / \langle \sigma \circ \iota \rangle \cong \mathbb{P}^1 $ given by $ (X:Y:Z) \mapsto (X:Y^2:Z)$.
A straightforward computation shows that the ramification points in $H_{a,b} / \langle \sigma \circ \iota \rangle$ of $\pi$ are $(0 : 0 : 1)$, $(0 : 1: 0)$ and $(A_i : 0 : 1)$ for $1 \leq i \leq 3$, whose ramification indexes are $2$, $6$ and $3$ respectively.

Consequently, $H_{a,b}$ is birational to
\[
y^{6} = (x- a_1)^{n_1} (x - a_2)^{n_2} (x - a_3)^{n_3} (x-a_4)^{n_4} (x-a_5)^{n_5} 
\]
for some $a_i \in k$ and some integers $1 \leq n_i < 6$ with $n_1 + n_2 + n_3 + n_4 + n_5 \equiv 0 \pmod{6}$.
Since the sequence of ramification indexes is given by $(v_1,v_2,v_3,v_4,v_5) = (6,3,2,2,2)$, we obtain $(n_1,n_2,n_3,n_4,n_5)=(1,2,3,3,3)$ or $(5,4,3,3,3)$ by $v_i = n / \mathrm{gcd}(n,n_i)$.
Therefore $H_{a,b}$ is a cyclic cover of type $(6; 1,2,3,3,3)$.

Here we prove (2).
Letting $\zeta_6$ be a primitive $6$-th root of unity in $k$, we decompose
\[
H^0 (H_{a,b}, \Omega_{H_{a,b}}^1) = \sum_{i=0}^5 D_i 
\]
as in Section \ref{subsec:cyclic}, where $D_i$ is the $\zeta_6^i$-eigenspace.
The dimension $d_i$ of $D_i$ for $1 \leq i \leq 5$ is
\[
d_i := \mathrm{dim} (D_i) = \frac{(i \bmod{6})}{6} + \frac{(2i \bmod{6})}{6} + \frac{(3i \bmod{6})}{6} + \frac{(3i \bmod{6})}{6} + \frac{(3i \bmod{6})}{6}  - 1.
\]
We then have
$d_0 = 0$, $d_1 = 1$, $d_2 = 0$, $d_3 = 1$, $d_4 = 0$, $d_5 = 2$, and
\[
\mathscr{C} (D_{p i \bmod{6}}) \subset D_i,
\]
For $p \equiv 1\pmod{6}$, we have $\mathscr{C} (D_i) \subset D_i$ for any $i$.
For $p \equiv 5\pmod{6}$, we have $\mathscr{C} (D_5) \subset D_1$, $\mathscr{C} (D_3) \subset D_3$, and $\mathscr{C}(D_1) \subset D_5$, as desired.
\end{proof}

\subsection{Main algorithm and its complexity}\label{subsec:mainalg}

Now, we construct an algorithm to enumerate superspecial hyperelliptic curves $H_{a,b}$.
For the efficiency, let us here restrict ourselves to the case where $a$ and $b$ blong to $\mathbb{F}_{p^2}$.

\begin{thm}\label{thm:main1}
Main Algorithm below outputs the $k$-isomorphism classes of all s.sp. hyperelliptic curves of the form \eqref{eq:Hab} with $a,b \in \mathbb{F}_{p^2}$ in time $\tilde{O}(p^4)$.
Moreover, if the gcd of resultants of non-zero entries of the Cartier-Manin matrix of $H_{a,b}$ has degree $O(p)$, the complexity becomes $\tilde{O}(p^3)$.

\paragraph{\it Main Algorithm.} For a prime $p \geq 7$ as the input, conduct the following:
\begin{enumerate}
\item Regarding $a$ and $b$ as variables, compute the Cartier-Manin matrix $M_{a,b}$ of $H_{a,b}$.
\item Collect all $(a,b)\in \mathbb{F}_{p^2}^2$ such that $H_{a,b}$ is a s.sp.\ hyperelliptic curve, as follows:
\begin{enumerate}
\item[2-1.] Compute the solutions $(a_0,b_0) \in \mathbb{F}_{p^2}^2$ to $M_{a,b} = 0$.
\item[2-2.] For each solution $(a_0,b_0)$ computed in Step 2-1, check if the equation $y^2 = f_{a,b}(x)$ in \eqref{eq:Hab} for $(a,b)=(a_0,b_0)$ defines a hyperelliptic curve, by computing $\mathrm{gcd}(f_{a,b},f_{a,b}')$.
\end{enumerate}
\item For each of $H_{a,b}$'s collected in Step 2, check the condition of Lemma \ref{lem:isomA4} to decide whether $\overline{\mathrm{Aut}}(H_{a,b}) \cong A_4$ or not.
Output one $H_{a,b}$ with $\overline{\mathrm{Aut}}(H_{a,b}) \cong A_4$ (if exists) and all of $H_{a,b}$'s such that $\overline{\mathrm{Aut}}(H_{a,b})$ is not isomorphic to $A_4$.
\end{enumerate}
\end{thm}

\begin{proof}
The correctness follows from Lemmas \ref{lem:isom}, \ref{lem:ssp}, \ref{lem:HabIsom} and \ref{lem:A4S4}.
The complexity of Step 1 is estimated as $O(p^3)$ by Lemma \ref{lem:HabCM} below.
The most generic and efficient method for Step 2-1 is the following resultant-based method:
\begin{enumerate}
    \item[(1)] Compute resultants of non-zero entries of $M_{a,b}$ with respect to $a$ (or $b$).
    \item[(2)] Compute the gcd in $\mathbb{F}_{p^2}[b]$ of the resultants and its roots in $\mathbb{F}_{p^2}$.
    \item[(3)] For each root $b_0$, evaluate it to $b$ in $M_{a,b}$, and then compute the gcd in $\mathbb{F}_{p^2}[a]$ of non-zero entries of $M_{a,b_0}$ and its roots $\mathbb{F}_{p^2}$.
\end{enumerate}
Since the degree of each non-zero entry of $M_{a,b}$ is $O(p)$ (both in $a$ and $b$), the resultants are computed in $\tilde{O}(p^3)$~\cite{HL21}, and their gcd in (2) is computed in $\tilde{O}(p^2)$.
Letting the degree of the gcd be $d = O(p^2)$, one can compute its roots in $\tilde{O}(d^2)$.
For each root $b_0$ ($O(d)$ possible choices), evaluate it to $b$ of $M_{a,b}$ in time $O(p)$, and compute the gcd of non-zero entries of $M_{a,b_0}$ in time $\tilde{O}(p)$.
The roots of the second gcd are also computed in $\tilde{O}(p^2)$.
Step 2-2 is clearly done in constant time.
Thus, the complexity of Step 2 is upper-bounded by $\tilde{O}(p^3 + d^2 + d p^2)$.
Note that the number of roots $(a_0,b_0)$ is $ O(p^2)$.

As for Step 3, checking the condition in Lemma \ref{lem:isomA4} is done in constant time for each root $(a_0,b_0)$, so that the complexity of Step 3 is $O(p^2)$.
\end{proof}

We remark that $d$ and the number of roots $(a_0,b_0)$ are both $O(p)$ in practice, see Tables \ref{table:1} -- \ref{table:4} in Section \ref{subsec:imp} below.
From this, the complexities of Steps 2 and 3 are expected to be $\tilde{O}(p^3)$ and $O(p)$, and in this case the total compexity of Main Algorithm is $\tilde{O}(p^3)$.

\begin{lem}\label{lem:HabCM}
The Cartier-Manin matrix $M_{a,b}$ in Step 1 is computed in time $O(p^3)$.
\end{lem}

\begin{proof}
Recall from Lemma \ref{lem:ssp} that each $(i,j)$-entry of $M_{a,b}$ is the $x^{ip-j}$-coefficient of $f^n$ with $f := f_{a,b}$ and $n:= (p-1)/2$. 
Putting $g = x^9 + x^6 + a x^3 + b$, it follows from $f^n = x^n g^n$ that $x^{ip-j}$-coefficient of $f^n$ is equal to the $x^k$-coefficient of $g^n$ for $k = \frac{(2i-1)p + (2j-1)}{2}$.
Moreover, since $g^n$ is a polynomial in $x^3$, the coefficient of $x^k$ in $g^n$ is zero if $k \not\equiv 0 \pmod{3}$.
Here, it follows from \eqref{eq:relation} that
\begin{equation}\label{eq:formula1}
\begin{split}
k b (g^n)_{k} = & (3 (n+1)-k) a (g^n)_{k-3} + (6(n+1)-k) (g^n)_{k-6}+ (9(n+1)-k) (g^n)_{k-9} ,
\end{split}
\end{equation}
for any $k=3, 6, 9, \ldots , 3p-3$, where we used $g_9 = g_6 = 1$, $g_3 = a$ and $g_0 =b$.
Hence, we can recursively compute $(g^n)_k$ for all $k \leq 3p-3$, starting from $(g^n)_0 = b^{\frac{p-1}{2}}$.
In particular, the coefficients $(g^n)_k$ with $k = \frac{(2i-1)p + (2j-1) }{2}$ for $1 \leq i \leq 3$ and $1 \leq j \leq 4$ are computed, since the maximal vaule $(5p-1)/2$ of such $k$ is less than $3p$.
The number of required iteration is at most $(5p-1)/2 = O(p)$.
The cost of computing each $(g^n)_k$ is $O(p^2)$.
Indeed, each $(g^n)_k$ is a polynomial in $a$ and $b$ of total degree $\leq n=\frac{p-1}{2}$, and thus the number of its non-zero terms is $\binom{n+2}{2} = O(p^2)$.
Therefore, the total cost of computing the $(i,j)$-entries of $M_{a,b}$ for $1 \leq i \leq 3$ and $1 \leq j \leq 4$ is $O(p^3)$.


Next, we consider to compute the $x^{k}$-coefficients of $g^n$ for $k = \frac{(2i-1)p + (2j-1)}{2}$ with $i = 4$ and $1 \leq j \leq 4$, namely, $\frac{7p-7}{2}$, $\frac{7p-5}{2}$, $\frac{7p-3}{2}$ and $\frac{7p-1}{2}$.
It follows from \eqref{eq:relation3} that
\begin{equation}\label{eq:formula2}
\begin{split}
(9n-k) (g^n)_{k} = &  (k +9) b (g^n)_{k+9} + (k+6 -3n) a (g^n)_{k+6}+(k+3 -6n )  (g^n)_{k+3}.
\end{split}
\end{equation}
For $\frac{7p-7}{2} \leq k \leq 9n = \frac{9(p-1)}{2}$, we have $9n - k \leq p-1$, and thus we can recursively compute $(g^n)_k$ for all $\frac{7p-7}{2} \leq k \leq 9 n$, starting from $(g^n)_{9n} = 1$.
In particular, the coefficients $(g^n)_k$ for $k = \frac{7p-7}{2}$, $\frac{7p-5}{2}$, $\frac{7p-3}{2}$ and $\frac{7p-1}{2}$ are computed.
Hence, the total cost of computing the $(i,j)$-entries of $M_{a,b}$ for $i=4$ and $1 \leq j \leq 4$ is $O(p^3)$, similarly to the case where $1 \leq i \leq 3$ and $1 \leq j \leq 4$.
\end{proof}

The following lemma implies that we need not classify $H_{a,b}$'s with $\overline{\rm Aut}(H_{a,b}) \cong \mathbb{Z}_3$ or $\mathbb{Z}_9$ collected in Step 2 into isomorphism classes:

\begin{lem}\label{lem:HabIsom}
If two hyperlliptic curves $H_{a,b}$ and $H_{a',b'}$ with reduced automorphism groups $\mathbb{Z}_3$ or $\mathbb{Z}_9$ are isomorphic, then $(a,b) = (a',b')$.
\end{lem}

\begin{proof}
Assume that there exists an isomorphicm $\rho : H_{a,b} \to H_{a',b'}$.
Recall from Lemma \ref{lem:isom} that $\rho$ is represented by $(P, \lambda) \in \mathrm{GL}_2(k) \times k^{\times}$ as in Lemma \ref{lem:isom}. 
Let $\zeta$ be a primitive $3$rd root of unity in $k$.
For an order-$3$ automorphism $\sigma : (x,y) \mapsto (\zeta x, \zeta^2 y)$ on $H_{a,b}$ represented by $A := \mathrm{diag}(\zeta,1)$, we set $\tau := \rho^{-1} \sigma \rho$;
\[
\begin{CD}
H_{a,b} @>{\sigma}>> H_{a,b} \\
@A{\rho}AA @VV{\rho^{-1}}V   \\
H_{a',b'} @>{\tau}>> H_{a',b'} .
\end{CD}
\]
Since $\tau$ also has order $3$ in $\overline{\rm Aut}(H)$, and since $\mathbb{Z}_9$ has a unique subgroup of order $3$, we have that $\tau$ is given by $(x,y) \mapsto (\zeta^i, \pm \zeta^{2i} y)$ for $i=1$ or $2$.
Thus, comparing the matrices corresponding to the both sides of $\tau = \rho^{-1} \sigma \rho$, we obtain an equation $P^{-1} A P = A^i$ in $\mathrm{PGL_2}(k)$, say
\[
\begin{pmatrix}
\zeta \alpha & \zeta \beta \\
\gamma & \delta
\end{pmatrix}
=
\mu
\begin{pmatrix}
\zeta^i \alpha & \beta \\
\zeta^i \gamma & \delta
\end{pmatrix}
\]
for some $\mu \in k^{\times}$.

If $\delta \neq 0$, then $\mu = 1$, $\beta = \gamma = 0$ and $i=1$.
In this case, $\rho$ is $(x,y) \mapsto (\alpha \delta^{-1} x, \delta^{-5} \lambda y)$, and thus $y^2 = f_{a,b}(x)$ is transformed into $\lambda^2 y^2 =  \alpha^{10} x^{10} + \alpha^{7} \delta^3 x^7 + a \alpha^4 \delta^6 x^4 + b \alpha \delta^9 x = \lambda^2 f_{a',b'}(x)$.
Therefore, $\alpha^{10} = \alpha^7 \delta^3 = \lambda^2$, and hence $\alpha = \zeta^j \delta$ for some $j$.
Since $\zeta^j \delta^{10} = \lambda^2$, it follows also from $a \zeta^j \delta^{10} = a' \lambda^2$ and $b \zeta^j \delta^{10} = b' \lambda^2$ that $a = a'$ and $b= b'$.

Suppose $\delta = 0$; then $\mu = \zeta$, $\alpha = 0$ and $i = 2$.
In this case, $\rho$ is $(x,y) \mapsto (\frac{\beta}{\gamma x}, \frac{\lambda y}{\gamma^5 x^5})$, and thus $y^2 = f_{a,b}(x)$ is transformed into $\lambda^2 y^2 =  \beta^{10} + \beta^{7} \gamma^3 x^3 + a \beta^4 \gamma^6 x^6 + b \beta \gamma^9 x^9 = \lambda^2 f_{a',b'}(x)$.
This is a contradiction since $f_{a',b'}$ has degree $10$.
\end{proof}

If $\overline{\rm Aut}(H_{a,b})$ is not isomorphic to $\mathbb{Z}_3$ nor $\mathbb{Z}_9$, it follows from Lemma \ref{lem:C3} that $\overline{\rm Aut}(H_{a,b}) \cong A_4$ and ${\rm Aut}(H_{a,b}) \cong \mathrm{SL}_2(\mathbb{F}_3)$.
In this case, by considering elements in $\mathrm{SL}_2(\mathbb{F}_3)$, there exists an order-$2$ element in $\overline{\mathrm{Aut}}(H_{a,b})$ whose order in $\mathrm{Aut}(H_{a,b})$ is $4$.
Any two of such $H_{a,b}$'s are isomorphic by Lemma \ref{lem:A4S4}, and one of them is detected by Lemma \ref{lem:isomA4} below; more generally, we have the following lemma, whose proof is essentially same as that of \cite[Theorem 3.1.1]{MK22}:

\begin{lem}\label{lem:isom3}
Let $H : y^2 = f(x)$ be a hyperelliptic curve of genus $g$ over $k$, where $f(x)$ is a separable polynomial over $k$ of degree $2g+2$.
Then, the following are equivalent:
\begin{enumerate}
    \item[(1)] $\overline{\mathrm{Aut}}(H)$ has an element $\sigma$ of order $2$ which has order $4$ as an element in $\mathrm{Aut}(H)$.
    \item[(2)] There exist roots $a_1$ and $a_2$ in $k$ of $f$ such that
    \[
    \left\{ \frac{a_i - a_2}{a_i - a_1}  : 3 \leq i \leq 2g  \right\} = \left\{ - \frac{a_i - a_2}{a_i - a_1}  : 3 \leq i \leq 2g  \right\},
    \]
    where $a_3, \ldots , a_{2g+2}$ are the other roots of $f$.
\end{enumerate}
\end{lem}

\begin{proof}
Assume (1).
By Proposition \ref{prop:aut}, there exists a hyperelliptic curve $H' : y^2 = f' (x)$ over $k$ and an isomorphism $\rho : H' \to H$ such that the automorphism $\tau := \rho^{-1} \sigma \rho$ of $H'$ is represented by $(\mathrm{diag}(-1, 1), \mu')  \in \mathrm{GL}_2(k) \times k^{\times}$;
\[
\begin{CD}
H @>{\sigma}>> H \\
@A{\rho}AA @VV{\rho^{-1}}V   \\
H' @>{\tau}>> H' ,
\end{CD}
\]
where $\mu'$ is an element in $k$ satisfying $(\mu')^2 = \pm 1$.
Since $\tau$ also has order $4$ in ${\rm Aut}(H)$, we have $(\mu')^2 = -1$ and thus $\mu' = \pm i$, where $i$ is an element in $k$ with $i^2 = -1$.
Moreover, it follows from $\tau \in \mathrm{Aut}(H)$ that $- f'(-x) = f'(x)$.
Therefore, $f'(x)$ is of degree $2g+1$ and is divided by $x$.

As for the form of a matrix representing $\rho$, we may assume from the proof of \cite[Theorem 3.1.1]{MK22} that it is either of the following:
\[
\textbf{(A):}\quad
\begin{pmatrix}
a_1 & a_2 \\
1 & 1 
\end{pmatrix}
\quad
\text{or}
\quad
\textbf{(B):}\quad
\begin{pmatrix}
1 & b_1 \\
0 & 1 
\end{pmatrix}
,
\]
where $a_1,a_2,b_1\in k$.
The case {\bf (B)} is impossible since $\rho$ does not send the point at infinity to itself.
In the case {\bf (A)}, the inverse map $\rho^{-1}$ is represented by 
\[
\begin{pmatrix}
a_1 & a_2 \\
1 & 1 
\end{pmatrix}
^{-1}=\frac{1}{a_1-a_2}
\begin{pmatrix}
1 & -a_2 \\
-1 & a_1 
\end{pmatrix}
.
\]
For a ramification point $(\alpha, 0)$ of $H$ with a root $\alpha$ of $f$, its image in $H'$ by $\rho^{-1}$ is $\left(-\frac{\alpha-a_2}{\alpha-a_1},0\right)$.
Since $H'$ has $0$ and $\infty$ as ramification points, we have that $a_1$ and $a_2$ are (mutually different) roots of $f$.
The condition $\tau \in \mathrm{Aut}(H)$ implies the assertion (2).

Conversely, if (2) holds, then we define $\rho$ as in {\bf (A)} with roots $a_1$ and $a_2$ of $f$, and then the image of $\rho$ is a hyperelliptic curve $H'$ with ramification points $\infty$, $(0,0)$ and $(\pm b_j, 0)$ for some $b_j \in k^{\times}$ with $1 \leq j \leq g$.
Therefore $H'$ has an automorphism $(x,y) \mapsto (-x, i y)$, as desired.
\end{proof}

As a particular case of Lemma \ref{lem:isom3}, we have the following:

\begin{lem}\label{lem:isomA4}
$\overline{\mathrm{Aut}}(H_{a,b}) \cong A_4$ if and only if there exist roots $a_1$ and $a_2$ of $f_{a,b}$ satisfying the condition (2) of Lemma \ref{lem:isom3} for $f = f_{a,b}$ and $g=4$.
\end{lem}

\begin{rem}
Here, we list possible variants of Main Algorithm with $q=p^2$, and provide upper-bounds of their complexities;
Our bounds for Main Algorithm in Theorem \ref{thm:main1} does not exceed each of the bounds below: 
\begin{enumerate}
\item Brute force on $(a,b) \in \mathbb{F}_{q}^2$.
For each $(a,b) \in \mathbb{F}_{q}^2$, test $\mathrm{gcd}(f,f') = 1$ or not in constant time.
If $\mathrm{gcd}(f,f') =1$, compute $M_{a,b}$ in time $\tilde{O}(\sqrt{p})$.
The total complexity is $\tilde{O}(q^2 \sqrt{p})$.

\item For each $a \in \mathbb{F}_{q}$, compute $M_{a,b}$ in time $O(p^2)$ keeping $b$ as a parameter.
Then brute force on $b$:
Test $\mathrm{gcd}(f,f') = 1$ or not in constant time.
If $\mathrm{gcd}(f,f') =1$, evaluate it to $M_{a,b}$ in time $\tilde{O}(p)$.
The total complexity is $O(q (p^2 +q) )$.

\item For each $a \in \mathbb{F}_{q}$, compute $M_{a,b}$ in time $O(p^2)$ keeping $b$ as a parameter.
Compute the gcd of non-zero entries of $M_{a,b}$ in time $\tilde{O}(p)$.
Compute the roots of the gcd in time $\tilde{O}(p^2)$.
For each root $b$, test $\mathrm{gcd}(f,f') = 1$ or not in constant time.
The total complexity is $\tilde{O}(q p^2 )$.
\item Compute $M_{a,b}$ in time $O(p^3)$ keeping $a$ and $b$ as parameters.
Then brute force on $(a,b)$:
Test $\mathrm{gcd}(f,f') = 1$ or not in constant time.
If $\mathrm{gcd}(f,f') =1$, evaluate it to $M_{a,b}$ in time $O(p)$.
The total complexity is $O(p^3 + q^2 p )$.
\item Compute $M_{a,b}$ in time $O(p^3)$ keeping $a$ and $b$ as parameters.
Then brute force on $a$:
Evaluate it to $M_{a,b}$ in time $O(p)$, and compute the gcd of non-zero entries of $M_{a,b}$ in time $\tilde{O}(p)$.
Compute the roots of the gcd in time $\tilde{O}(p^2)$.
For each root $b$, test $\mathrm{gcd}(f,f') = 1$ or not in constant time.
The total complexity is $\tilde{O}(p^3 + q p^2 )$.
\end{enumerate}
\end{rem}

\subsection{Implementation and computational results}\label{subsec:imp}

We implemented Main Algorithm on Magma V2.26-10 on a PC with macOS Monterey 12.0.1, at 2.6 GHz CPU 6 Core (Intel Core i7) and 16GB memory (cf.\ \cite{HPkudo} for the source code ``NKT\_enum3.txt'').
Executing the implemented algorithm, we obtain Theorem \ref{thm:main22}.
Our computational results are summarized in Tables \ref{table:1} -- \ref{table:4} below.



We can easily increase the upper bound on $p$ in Theorem \ref{thm:main22}.
For example, on the PC described above, computing the $\overline{\mathbb{F}_p}$-isomorphicm classes of s.sp.\ $H_{a,b}$'s with $a,b \in \mathbb{F}_{p^2}$ for all $17 \leq p < 1000$ took 13,226 seconds (about 3.7 hours) in total, and the execution for $p=997$ took only 196 seconds.

\begin{table}[H]
\centering{
\caption{Computational results for $17 \leq p < 200$ obtained by the execution of Main Algorithm in Theorem \ref{thm:main1}.
``Num.\ of $H_{a,b}$'' denotes the number of $\overline{\mathbb{F}_p}$-isomorphism classes of obtained $H_{a,b}$'s.
The time taken in each step (e.g., ``Step 1'' in the table means the time taken for Step 1 of Main Algorithm) and the total time (denoted by ``Total time'') are also shown in seconds.
``Deg.\ of gcd." denotes the degree of the gcd of the resultants computed in Step 2.}
\label{table:1}
\scalebox{1}{
\begin{tabular}{c||c|c|c||r|r|r||r} \hline
$p$ & $p \bmod{3}$ & Num.\ of $H_{a,b}$ & Deg.\ of gcd & Step 1 & Step 2 & Step 3 & Total time \\ \hline
17 & 2 & 1 & 1 &    0.010 & $<$ 0.020 & $<$ 0.001 & 0.020 \\ \hline
19 & 1 & 0 & 0 &    0.010 & 0.010 & - & 0.010 \\ \hline
23 & 2 & 2 & 2 &    < 0.001 & 0.010 & < 0.001 & 0.010 \\ \hline
29 & 2 & 1 & 1 &    < 0.001 & 0.020 & < 0.001 & 0.020 \\ \hline
31 & 1 & 0 & 0 &    < 0.001 & 0.020 & - & 0.020 \\ \hline
37 & 1 & 0 & 0 &    < 0.001 & 0.010 & - & 0.010 \\ \hline
41 & 2 & 4 & 4 &    < 0.001 & < 0.001 & 0.010 & 0.010 \\ \hline
43 & 1 & 0 & 0 &    < 0.001 & < 0.001 & - & < 0.001 \\ \hline
47 & 2 & 5 & 5 &    < 0.001 & 0.010 & 0.010 & 0.010 \\ \hline
53 & 2 & 4 & 6 &    < 0.001 & 0.010 & <0.001 & 0.010 \\ \hline
59 & 2 & 6 & 9 &    < 0.001 & 0.010 & 0.010 & 0.020 \\ \hline
61 & 1 & 0 & 0 &    < 0.001 & 0.001 & - & 0.010 \\ \hline
67 & 1 & 0 & 0 &    < 0.001 & 0.010 & - & 0.010 \\ \hline
71 & 2 & 9 & 10 &    < 0.001 & 0.010 & 0.010 & 0.020 \\ \hline
73 & 1 & 0 & 0 &    < 0.001 & 0.010 & - & 0.010 \\ \hline
79 & 1 & 0 & 0 &    0.010 & 0.010 & - & 0.020 \\ \hline
83 & 2 & 8 & 11 &    < 0.001 & 0.020 & 0.020 & 0.040 \\ \hline
89 & 2 & 7 & 7 &    < 0.001 & 0.030 & 0.010 & 0.030 \\ \hline
97 & 1 & 0 & 0 &    0.010 & 0.030 & - & 0.040 \\ \hline
101 & 2 & 8 & 10 &   0.010 & 0.060 & 0.020 & 0.090 \\ \hline
103 & 1 & 0 & 0 &   < 0.001 & 0.010 & - & 0.010 \\ \hline
107 & 2 & 4 & 5 &   0.010 & 0.020 & 0.010 & 0.040 \\ \hline
109 & 1 & 0 & 0 &   < 0.001 & 0.010 & - & 0.010 \\ \hline
113 & 2 & 14 & 15 &  < 0.001 & 0.020 & 0.030 & 0.050 \\ \hline
127 & 1 & 0 & 0 &    < 0.001 & 0.030 & - & 0.030 \\ \hline
131 & 2 & 18 & 22 &    0.010 & 0.040 & 0.040 & 0.80 \\ \hline
137 & 2 & 12 & 13 &   0.010 & 0.040 & 0.020 & 0.080 \\ \hline
139 & 1 & 0 & 0 &    < 0.001 & 0.040 & - & 0.040 \\ \hline
149 & 2 & 18 & 22 &    0.010 & 0.060 & 0.060 & 0.130 \\ \hline
151 & 1 & 0 & 0 &    0.010 & 0.050 & - & 0.060 \\ \hline
157 & 1 & 0 & 0 &    0.010 & 0.060 & - & 0.070 \\ \hline
163 & 1 & 0 & 0 &    0.010 & 0.070 & - & 0.080 \\ \hline
167 & 2 & 26 & 29 &    0.010 & 0.100 & 0.050 & 0.160 \\ \hline
173 & 2 & 22 & 22 &    0.020 & 0.110 & 0.060 & 0.190 \\ \hline
179 & 2 & 17 & 17 &    0.020 & 0.130 & 0.040 & 0.180 \\ \hline
181 & 1 & 0 & 0 &    0.010 & 0.100 & - & 0.120 \\ \hline
191 & 2 & 32 & 33 &    0.020 & 0.180 & 0.050 & 0.250 \\ \hline
193 & 1 & 0 & 0 &    0.020 & 0.140 & - & 0.160 \\ \hline
197 & 2 & 21 & 24 &    0.020 & 0.140 & < 0.001 & 0.160 \\ \hline
199 & 1 & 0 & 0 &    0.020 & 0.150 & - & 0.170 \\ \hline
\end{tabular}
}
}
\end{table}

\begin{table}[H]
\centering{
\caption{Computational results for $200 \leq p < 450$ obtained by the execution of Main Algorithm in Theorem \ref{thm:main1}.
The notations are same as in Table \ref{table:1}.}
\label{table:2}
\scalebox{1}{
\begin{tabular}{c||c|c|c||r|r|r||r} \hline
$p$ & $p \bmod{3}$ & Num.\ of $H_{a,b}$ & Deg.\ of gcd & Step 1 & Step 2 & Step 3 & Total time \\ \hline
211 & 1 & 0 & 0 &    0.020 & 0.210 & - & 0.230 \\ \hline
223 & 1 & 0 & 0 &    0.030 & 0.270 & - & 0.290 \\ \hline
227 & 2 & 29 & 35 & 0.020 & 0.350 & 0.090 & 0.470 \\ \hline
229 & 1 & 0 & 0 &    0.020 & 0.300 & - & 0.320 \\ \hline
233 & 2 & 30 & 31 & 0.020 & 0.410 & 0.080 & 0.510 \\ \hline
239 & 2 & 36 & 40 & 0.030 & 0.450 & 0.070 & 0.550 \\ \hline
241 & 1 & 0 & 0 & 0.020 & 0.390 & - & 0.420 \\ \hline
251 & 2 & 28 & 31 & 0.030 & 0.570 & 0.060 & 0.660 \\ \hline
257 & 2 & 28 & 31 & 0.040 & 0.640 & 0.080 & 0.770 \\ \hline
263 & 2 & 58 & 60 & 0.030 & 0.720 & 0.150 & 0.910 \\ \hline
269 & 2 & 32 & 35 & 0.040 & 0.770 & 0.070 & 0.880 \\ \hline
271 & 1 & 0 & 0 & 0.050 & 0.650 & - & 0.700 \\ \hline
277 & 1 & 0 & 0 & 0.050 & 0.700 & - & 0.750 \\ \hline
281 & 2 & 34 & 37 & 0.060 & 0.960 & 0.080 & 1.110 \\ \hline
283 & 1 & 0 & 0 & 0.050 & 0.790 & - & 0.850 \\ \hline
293 & 2 & 29 & 31 & 0.060 & 1.170 & 0.080 & 1.310 \\ \hline
307 & 1 & 0 & 0 & 0.070 & 1.150 & - & 1.220 \\ \hline
311 & 2 & 62 & 69 & 0.070 & 1.150 & 0.120 & 1.770 \\ \hline
313 & 1 & 0 & 0 & 0.070 & 1.290 & - & 1.360 \\ \hline
317 & 2 & 21 & 24 & 0.060 & 1.670 & 0.070 & 1.800 \\ \hline
331 & 1 & 0 & 0 & 0.070 & 1.630 & - & 1.700 \\ \hline
337 & 1 & 0 & 0 & 0.080 & 1.780 & - & 1.860 \\ \hline
347 & 2 & 61 & 70 & 0.090 & 2.590 & 0.180 & 2.870 \\ \hline
349 & 1 & 0 & 0 & 0.090 & 2.100 & - & 2.200 \\ \hline
353 & 2 & 25 & 35 & 0.100 & 2.860 & 0.060 & 3.020 \\ \hline
359 & 2 & 55 & 61 & 0.110 & 3.030 & 0.090 & 3.230 \\ \hline
367 & 1 & 0 & 0 & 0.120 & 2.680 & - & 2.810 \\ \hline
373 & 1 & 0 & 0 & 0.120 & 2.870 & - & 3.000 \\ \hline
379 & 1 & 0 & 0 & 0.150 & 3.120 & - & 3.270 \\ \hline
383 & 2 & 72 & 75 & 0.130 & 4.110 & 0.150 & 4.390 \\ \hline
389 & 2 & 49 & 50 & 0.130 & 4.380 & 0.150 & 4.670 \\ \hline
397 & 1 & 0 & 0 & 0.140 & 3.760 & - & 3.900 \\ \hline
401 & 2 & 44 & 48 & 0.150 & 5.020 & 0.120 & 5.290 \\ \hline
409 & 1 & 0 & 0 & 0.150 & 4.330 & - & 4.490 \\ \hline
419 & 2 & 61 & 63 & 0.180 & 6.070 & 0.160 & 6.410 \\ \hline
421 & 1 & 0 & 0 & 0.170 & 4.850 & - & 5.020 \\ \hline
431 & 2 & 72 & 79 & 0.200 & 6.910 & 0.140 & 7.260 \\ \hline
433 & 1 & 0 & 0 & 0.170 & 5.580 & - & 5.760 \\ \hline
439 & 1 & 0 & 0 & 0.190 & 5.840 & - & 6.020 \\ \hline
443 & 2 & 50 & 52 & 0.190 & 7.780 & 0.150 & 8.130 \\ \hline
449 & 2 & 38 & 44 & 0.200 & 8.140 & 0.110 & 8.450 \\ \hline

\end{tabular}
}
}
\end{table}

\begin{table}[H]
\centering{
\caption{Computational results for $450 \leq p < 750$ obtained by the execution of Main Algorithm in Theorem \ref{thm:main1}.
The notations are same as in Table \ref{table:1}.}
\label{table:3}
\scalebox{1}{
\begin{tabular}{c||c|c|c||r|r|r||r} \hline
$p$ & $p \bmod{3}$ & Num.\ of $H_{a,b}$ & Deg.\ of gcd & Step 1 & Step 2~~ & Step 3~~ & Total time \\ \hline
457 & 1 & 0 & 0 & 0.210 & 6.980 & - & 7.200 \\ \hline
461 & 2 & 54 & 59 & 0.240 & 9.330 & 0.140 & 9.690 \\ \hline
463 & 1 & 0 & 0 & 0.220 & 7.450 & - & 7.680 \\ \hline
467 & 2 & 73 & 78 & 0.220 & 9.750 & 0.180 & 10.140 \\ \hline
479 & 2 & 82 & 93 & 0.240 & 10.950 & 0.180 & 11.370 \\ \hline
487 & 1 & 0 & 0 & 0.290 & 9.130 & - & 9.420 \\ \hline
491 & 2 & 79 & 84 & 0.300 & 12.200 & 0.200 & 12.710 \\ \hline
499 & 1 & 0 & 0 & 0.270 & 10.290 & - & 10.570 \\ \hline
503 & 2 & 93 & 105 & 0.280 & 13.560 & 0.210 & 14.060 \\ \hline
509 & 2 & 59 & 64 & 0.290 & 14.100 & 0.190 & 14.590 \\ \hline
521 & 2 & 70 & 78 & 0.300 & 15.690 & 0.200 & 16.200 \\ \hline
523 & 1 & 0 & 0 & 0.310 & 12.330 & - & 12.650 \\ \hline
541 & 1 & 0 & 0 & 0.350 & 14.270 & - & 14.630 \\ \hline
547 & 1 & 0 & 0 & 0.350 & 15.240 & - & 15.600 \\ \hline
557 & 2 & 67 & 68 & 0.380 & 20.610 & 0.190 & 21.190 \\ \hline
563 & 2 & 75 & 84 & 0.380 & 21.600 & 0.190 & 22.170 \\ \hline
569 & 2 & 78 & 87 & 0.390 & 22.660 & 0.210 & 23.270 \\ \hline
571 & 1 & 0 & 0 & 0.400 & 18.010 & - & 18.410 \\ \hline
577 & 1 & 0 & 0 & 0.420 & 19.000 & - & 19.430 \\ \hline
587 & 2 & 89 & 94 & 0.430 & 26.010 & 0.240 & 26.680 \\ \hline
593 & 2 & 94 & 101 & 0.450 & 26.740 & 0.280 & 27.470 \\ \hline
599 & 2 & 108 & 122 & 0.460 & 28.100 & 0.240 & 28.810 \\ \hline
601 & 1 & 0 & 0 & 0.460 & 22.440 & - & 22.910 \\ \hline
607 & 1 & 0 & 0 & 0.480 & 23.300 & - & 23.790 \\ \hline
613 & 1 & 0 & 0 & 0.490 & 24.220 & - & 24.720 \\ \hline
617 & 2 & 60 & 61 & 0.510 & 32.220 & 0.200 & 32.940 \\ \hline
619 & 1 & 0 & 0 & 0.510 & 25.350 & - & 25.870 \\ \hline
631 & 1 & 0 & 0 & 0.630 & 27.620 & - & 28.250 \\ \hline
641 & 2 & 84 & 95 & 0.680 & 37.300 & 0.240 & 38.230 \\ \hline
643 & 1 & 0 & 0 & 0.670 & 29.660 & - & 30.340 \\ \hline
647 & 2 & 106 & 118 & 0.670 & 38.640 & 0.250 & 39.570 \\ \hline
653 & 2 & 54 & 56 & 0.690 & 39.980 & 0.160 & 40.840 \\ \hline
659 & 2 & 89 & 99 & 0.700 & 41.270 & 0.280 & 42.260 \\ \hline
661 & 1 & 0 & 0 & 0.710 & 33.170 & - & 33.880 \\ \hline
673 & 1 & 0 & 0 & 0.850 & 36.200 & - & 37.060 \\ \hline
677 & 2 & 59 & 63 & 0.880 & 46.260 & 0.190 & 47.340 \\ \hline
683 & 2 & 102 & 106 & 0.780 & 48.000 & 0.310 & 49.100 \\ \hline
691 & 1 & 0 & 0 & 0.810 & 40.050 & - & 40.870 \\ \hline
701 & 2 & 68 & 78 & 0.940 & 53.350 & 0.210 & 54.510 \\ \hline
709 & 1 & 0 & 0 & 0.790 & 44.360 & - & 45.160 \\ \hline
719 & 2 & 112 & 122 & 1.010 & 59.640 & 0.250 & 60.910 \\ \hline
727 & 1 & 0 & 0 & 0.920 & 49.700 & - & 50.630 \\ \hline
733 & 1 & 0 & 0 & 0.940 & 3820.520 & - & 3821.470 \\ \hline
739 & 1 & 0 & 0 & 4.720 & 2682.300 & - & 2687.080 \\ \hline
743 & 2 & 124 & 131 & 4.770 & 158.440 & 0.510 & 163.780 \\ \hline
\end{tabular}
}
}
\end{table}

\begin{table}[H]
\centering{
\caption{Computational results for $750 \leq p < 1000$ obtained by the execution of Main Algorithm in Theorem \ref{thm:main1}.
The notations are same as in Table \ref{table:1}.}
\label{table:4}
\vspace{-2pt}
\scalebox{0.93}{
\begin{tabular}{c||c|c|c||r|r|r||r} \hline
$p$ & $p \bmod{3}$ & Num.\ of $H_{a,b}$ & Deg.\ of gcd & Step 1 & Step 2~~ & Step 3~ & Total time \\ \hline

751 & 1 & 0 & 0 & 1.290 & 76.070 & - & 77.370 \\ \hline
757 & 1 & 0 & 0 & 1.260 & 79.710 & - & 80.990 \\ \hline
761 & 2 & 129 & 137 & 1.330 & 102.390 & 0.560 & 104.300 \\ \hline
769 & 1 & 0 & 0 & 1.330 & 85.700 & - & 87.050 \\ \hline
773 & 2 & 106 & 109 & 1.390 & 108.570 & 0.490 & 110.470 \\ \hline
787 & 1 & 0 & 0 & 1.480 & 94.960 & - & 96.460 \\ \hline
797 & 2 & 90 & 95 & 1.530 & 125.590 & 0.430 & 127.570 \\ \hline
809 & 2 & 94 & 107 & 1.690 & 136.260 & 0.390 & 138.360 \\ \hline
811 & 1 & 0 & 0 & 1.560 & 105.830 & - & 107.400 \\ \hline
821 & 2 & 107 & 116 & 1.830 & 141.670 & 0.470 & 143.990 \\ \hline
823 & 1 & 0 & 0 & 1.610 & 113.860 & - & 115.490 \\ \hline
827 & 2 & 120 & 132 & 1.680 & 147.850 & 0.530 & 150.080 \\ \hline
829 & 1 & 0 & 0 & 1.690 & 116.740 & - & 118.450 \\ \hline
839 & 2 & 119 & 131 & 1.900 & 155.280 & 0.360 & 157.550 \\ \hline
853 & 1 & 0 & 0 & 1.840 & 132.260 & - & 134.120 \\ \hline
857 & 2 & 121 & 124 & 1.880 & 152.210 & 0.410 & 154.520 \\ \hline
859 & 1 & 0 & 0 & 1.400 & 97.910 & - & 99.320 \\ \hline
863 & 2 & 138 & 151 & 1.430 & 127.050 & 0.370 & 128.870 \\ \hline
877 & 1 & 0 & 0 & 1.510 & 107.420 & - & 108.940 \\ \hline
881 & 2 & 112 & 119 & 1.530 & 138.870 & 0.350 & 140.760 \\ \hline
883 & 1 & 0 & 0 & 1.540 & 110.550 & - & 112.100 \\ \hline
887 & 2 & 156 & 161 & 1.580 & 142.260 & 0.400 & 144.250 \\ \hline
*907 & 1 & 0 & 1 & 1.640 & 122.450 & 0.030 & 124.140 \\ \hline
911 & 2 & 182 & 198 & 1.680 & 163.200 & 0.440 & 165.340 \\ \hline
919 & 1 & 0 & 0 & 1.710 & 130.080 & - & 131.810 \\ \hline
929 & 2 & 89 & 92 & 2.050 & 172.330 & 0.240 & 174.610 \\ \hline
937 & 1 & 0 & 0 & 1.820 & 140.490 & - & 142.330 \\ \hline
941 & 2 & 126 & 139 & 1.840 & 181.600 & 0.400 & 183.860 \\ \hline
947 & 2 & 95 & 103 & 1.890 & 185.830 & 0.320 & 188.060 \\ \hline
953 & 2 & 109 & 114 & 2.280 & 192.910 & 0.340 & 195.540 \\ \hline
967 & 1 & 0 & 0 & 2.010 & 162.060 & - & 164.090 \\ \hline
971 & 2 & 122 & 133 & 2.290 & 206.790 & 0.360 & 210.190 \\ \hline
977 & 2 & 120 & 123 & 2.320 & 212.800 & 0.420 & 215.550 \\ \hline
983 & 2 & 115 & 125 & 2.360 & 218.530 & 0.360 & 221.270 \\ \hline
991 & 1 & 0 & 0 & 2.450 & 184.370 & - & 186.840 \\ \hline
997 & 1 & 0 & 0 & 2.590 & 192.880 & - & 195.500 \\ \hline
\end{tabular}
}
}
\end{table}
\vspace{-5mm}

For each $p$ with $p \equiv 2 \pmod{3}$ in the tables, a s.sp.\ $H_{a,b}$ with $(a,b) \in \mathbb{F}_{p^2}^2$ do exists.
On the other hand, for any $p$ with $p \equiv 1 \pmod{3}$, there is no $(a,b) \in \mathbb{F}_{p^2}^2$ such that $H_{a,b}$ is a s.sp. hyperelliptic curve.
The degree of the gcd of the resultants computed in Step 2 and the number of isomorphism classes of obtained $H_{a,b}$'s might follow $O(p)$, which implies that the complexity of Step 2 is $\tilde{O}(p^3)$ in practice, see Theorem \ref{thm:main1}.

Most of time is spent at Step 2.
We see from the tables that Steps 1 and 2 might follow our estimations $O(p^3)$ and $\tilde{O}(p^3)$ respectively.
As for Step 3, our estimation in the proof of Theorem \ref{thm:main1} is $O(p^2)$, but it might be $O(p)$ in practice since the number of $(a,b)$ for which $H_{a,b}$ is a s.sp.\ hyperelliptic curve would be $O(p)$.
Note that for some $p$, strangely it took a long time for Step 1 or Step 2 compared with other cases (due to memory issue?): $p=733$, $739$ and $743$.

%% file: section4.tex
\section{Concluding remarks}

We realized an algorithm with complexity $\tilde{O}(p^4)$ in theory but $\tilde{O}(p^3)$ in practice, specific to producing s.sp.\ hyperelliptic curves of genus $4$, restricting to a parametric family of curves $H_{a,b}$ given by $y^2 = x^{10} + x^7 + a x^4 + b x$.
Our case is included in the case where $\overline{\rm Aut}(H) \supset \mathbb{Z}_3$ in Theorem \ref{thm:app}, while a recent work~\cite{OKH22} presented at WAIFI2022 treats the case where $\overline{\rm Aut}(H) \supset V_4:=\mathbb{Z}_2 \times \mathbb{Z}_2$ (Klein 4-group).
Our algorithm cannot enumerate all s.sp.\ hyperelliptic curves of genus $4$ different from the algorithm in \cite{KH18} at WAIFI2018, but it is expected from Theorem \ref{thm:main22} to surely find such a curve for arbitrary $p \geq 17$ with $p \equiv 2 \bmod{3}$.
By executing the algorithm on Magma, we succeeded in enumerating s.sp.\ hyperelliptic curves $H_{a,b}$ with $(a,b) \in \mathbb{F}_{p^2}^2$ for every $p$ between $17$ to $1000$, which is much larger than $p=17,19$ as in the enumeration of \cite{KH18}.

A future work is to construct an algorithm for finding s.sp.\ hyperelliptic curves of genus $4$ in the case where $p \equiv 1 \bmod{3}$, with complexity lower than the algorithm of \cite{KH18}.
(Focusing on the case where $\overline{\rm Aut}(H) \cong \mathbb{Z}_4$ in Theorem \ref{thm:app} might provide a solution.)
Extending our method to higher genus cases is also an interesting problem.

\subsection*{Acknowledgement}
The authors thank Shushi Harashita and Ryo Ohashi for hepful comments.
This work was supported by JSPS Grantin-Aid for Young Scientists 20K14301 and JST CREST Grant Number JPMJCR2113.

%% file: section5.tex
\appendix

\section{Automorphism groups of hyperelliptic curves of genus four}\label{sec:app}

In this section, we classify hyperelliptic curves $H$ of genus $4$ over an algebraically closed field $k$ of characteristic $p \geq 7$ in terms of their automorphism groups $\mathrm{Aut}(H)$ as finite groups.
Such a classification of hyperelliptic curves of given or arbitrary genus is classically an important task in algebraic geometry, and the case of characteristic zero has been completely solved by Shaska, see \cite{Shaska03} and \cite{Shaska04}. 
He applied a classification of finite subgroups of the projective linear group $\mathrm{PGL}_2$ to determining $\overline{\mathrm{Aut}}(H)$ which is canonically embedded into $\mathrm{PGL}_2$.
After determining possible types of $\overline{\mathrm{Aut}}(H)$, he found an equation (in reduced form) defining $H$ for each type and the structure of $\mathrm{Aut}(H)$, with the action of elements in $\overline{\mathrm{Aut}}(H)$ as matrices.
His idea can be applied to the positive characteristic case, with carefully considering some exceptional cases depending on $p$ and $g$ such as the existence of $p$-subgroups of $\mathrm{PGL}_2$, see Proposition \ref{prop:Faber} below.
For the case of genus $2$ and $3$, explicit classifications in characteristic $p$ are given in \cite{Igusa} (and \cite{IKO}) and \cite[Table 2]{LRS} respectively.

Analogously to \cite{Shaska03} in the characteristic zero case and \cite{LRS} in the genus-$3$ case, we shall prove the below theorem (Theorem \ref{thm:app}) for the genus-$4$ case.
This result is referred in the main contents of this paper to characterize superspecial curves obtained in previous and this works.
Some lemmas to prove Theorem \ref{thm:app} are also used in Section \ref{subsec:mainalg} to make our algorithm quite efficient.

\begin{theor}\label{thm:app}
Assume that $p \geq 7$.
The reduced automorphism group $\overline{\mathrm{Aut}(H)}$ of a hyperelliptic curve $H$ of genus $4$ over an algebraically closed field $k$ of characteristic $p$ is isomorphic to either of the $10$ finite groups listed in Table \ref{table:aut}.
In each type of $\overline{\mathrm{Aut}(H)}$, the hyperelliptic curve $H$ is isomorphic to $y^2 = f(x)$ given in the forth column of Table \ref{table:aut}, and the finite group isomorphic to $\mathrm{Aut}(H)$ is provided in the fifth column of the table.
\end{theor}

\renewcommand{\arraystretch}{1.4}
\begin{table}[H]
\centering{
\caption{Possible finite groups isomorphic to $\overline{\rm Aut}(H)$ for hyperelliptic curves $H$ of genus $4$ over an algebraically closed field $k$ of characteristic $p \geq 7$, and hyperelliptic equations $y^2 = f(x)$ defining $H$, where $a,b,c,d \in k$ and where $i$ is an element satisfying $i^2 = -1$ in $k$.
We denote by ``dim.'' the dimension of the moduli space of $H$ with ${\rm Aut}(H)$ of given type.}
\label{table:aut}
\scalebox{0.85}{
\begin{tabular}{c||c|c|l|c|c|c} \hline
Type & $\overline{\mathrm{Aut}}(H)$ & $\# \overline{\mathrm{Aut}}(H)$ & $y^2 = f(x)$ birational to $H$ & ${\mathrm{Aut}}(H)$ & $\# {\mathrm{Aut}}(H)$ & dim.\ \\ \hline 
{\bf 1} & $\{ 0 \}$ & 1 & $y^2 = (\mbox{square-free polynomial in $x$ of degree 9 or 10})$   & $\mathbb{Z}_2$ & 2 & 7 \\ \hline
{\bf 2-1} & $\mathbb{Z}_2$ & 2 & $y^2 = x^{10} + a x^8 + b x^6 + c x^4 + d x^2 + 1$   & $V_4$ & 4 & 4 \\ \hline
{\bf 2-2} & $\mathbb{Z}_2$ & 2 & $y^2 = x^{9} + a x^7 + b x^5 + c x^3 + x$ & $\mathbb{Z}_4$ & 4 & 3 \\ \hline
{\bf 3} & $\mathbb{Z}_3$ & 3 & $y^2 = x^{10} + a x^7 + b x^4 + x$  & $\mathbb{Z}_6$ & 6 & \multirow{3}{*}{2} \\ 
\multirow{2}{*}{\bf 4-1} & \multirow{2}{*}{$V_4$} & \multirow{2}{*}{4} & $y^2 = x^{10} + a x^8 + b x^6 + b x^4 + a x^2 + 1$, or  & \multirow{2}{*}{$D_4$} & 8 &  \\ 
 &  &  & $y^2 = x^{9} + a x^7 + b x^5 + a x^3 + x$ &  & 8 &  \\ \hline
{\bf 4-2} & $V_4$ & 4 & $y^2 = x (x^4-1)(x^4+a x^2+1)$ & $Q_8$ & 8 & \multirow{3}{*}{1}\\ 
{\bf 5} & $D_4$ & 8 & $y^2 = x^{9} + a x^5 + x$ & $D_8$ & 16 &  \\ 
{\bf 6} & $D_5$ & 10 & $y^2 = x^{10} + a x^5 + 1$ & $D_{10}$ & 20 &  \\ \hline
{\bf 7} & $A_4$ & 12 & $y^2 = x (x^4-1)(x^4 + 2i \sqrt{3} x^2 + 1)$ & $\mathrm{SL}_2(\mathbb{F}_3)$ & 24 & \multirow{4}{*}{0}\\ 
{\bf 8} & $D_8$ & 16 & $y^2 = x^{9} + x$ & $\mathbb{Z}_{16} \rtimes \mathbb{Z}_2$ & 32 &  \\ 
{\bf 9} & $\mathbb{Z}_9$ & 9 & $y^2 = x^{10} + x$ & $\mathbb{Z}_{18}$  & 18 &  \\
{\bf 10} & $D_{10}$ & 20 & $y^2 = x^{10}+1$ & $ \mathbb{Z}_5 \rtimes D_{4} $ & 40 &   \\ \hline
\end{tabular}
}
}
\end{table}
\renewcommand{\arraystretch}{1}


\renewcommand{\arraystretch}{1}

\begin{figure}[H]
\label{fig:aut}
 \centering
\caption{Dimensions and containment relationships among the moduli spaces of genus-$4$ hyperelliptic curves with given automorphism groups ${\rm Aut}(H)$.}
\renewcommand{\arraystretch}{0.8}
{\small 
\[
\xymatrix{
 & &  \mathbb{Z}_2  \ar@{-}[rd] \ar@{-}[dd] \ar@{-}[llddd]  & & & \mbox{7-dimensional}\\
 & & & V_4 \ar@{-}[dd]  \ar@{-}[rddd]&   & \mbox{4-dimensional}\\
 & & \mathbb{Z}_4 \ar@{-}[dd] \ar@{-}[rd]   & & & \mbox{3-dimensional}\\
 \mathbb{Z}_6 \ar@{-}[dd] \ar@{-}[rdd]  & &  \ar@{-}[d] & D_4 \ar@{-}[rdd] \ar@{-}[d] &  & \mbox{2-dimensional}\\
 & & Q_8 \ar@{-}[ld]\ar@{-}[rd]  &  D_8 \ar@{-}[d]    & D_{10} \ar@{-}[d] & \mbox{1-dimensional}\\
 \mathbb{Z}_{18} & \mathrm{SL}_2(\mathbb{F}_3) & & \mathbb{Z}_{16} \rtimes \mathbb{Z}_2  & \mathbb{Z}_5 \rtimes D_4 & \mbox{0-dimensional}
}
\] 
}
\renewcommand{\arraystretch}{1}
\end{figure}
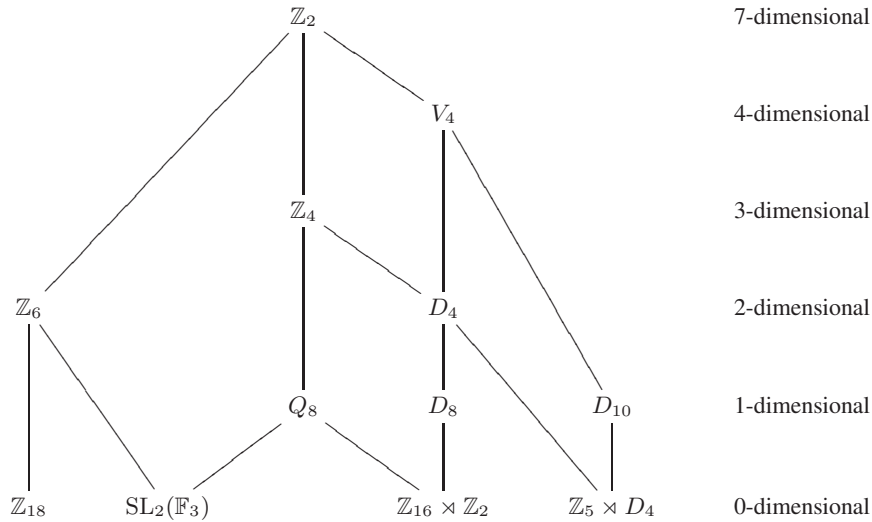

\subsection{Reduced automorphism groups of hyperelliptic curves}

Let $H : y^2 = f(x)$ be a hyperelliptic curve of genus $g$ over an algebraically closed field $k$ of characteristic $p \geq 0$ with $p \neq 0$.
As a particular case of Lemma \ref{lem:isom}, 
we have a group isomorphism
\[
G_f := \{ (P, \lambda ) \in \mathrm{GL}_2(k) \times k^{\times} : P.f = \lambda^2 f(x) \} / \{ (\mu I_2, \mu^{g+1}) : \mu \in I_2 \} \cong \mathrm{Aut}(H)
\]
sending $[P, \lambda ]$ to $(x,y) \mapsto \left(  \frac{\alpha x + \beta}{\gamma x + \delta}, \frac{\lambda y}{(\gamma x + \delta)^{g+1}} \right)$ for $P = \left( \begin{smallmatrix}
\alpha & \beta \\
\gamma & \delta
\end{smallmatrix} \right)$, where $P.f := (\gamma x + \delta)^{2g+2} f \left(  \frac{\alpha x + \beta}{\gamma x + \delta} \right)$ and where $[P, \lambda]$ denotes the equivalence class of $(P, \lambda)$ in the left hand side of the isomorphism.
In the following, we identify ${\rm Aut}(H)$ with $G_f$ via the above isomorphism.
We also note that the reduced automorpism group $\overline{\mathrm{Aut}}(H) \cong G_f / \langle [I_2, -1] \rangle$ is isomorphic to a subgroup of the projective linear group $\mathrm{PGL}_2(k)$, where an injective homomorphism $\overline{\mathrm{Aut}}(H) \cong G_f / \langle [I_2, -1] \rangle \to \mathrm{PGL}_2(k)$ is induced from $\mathrm{Aut}(H) \to \mathrm{PGL}_2(k) \ ; \ [P,\lambda] \to P$.
More explicitly, we have
\begin{equation}\label{eq:RedAut}
    \overline{\mathrm{Aut}}(H) \cong \overline{G}_{f} := \{ P \in \mathrm{PGL}_2(k): P.f(x) = \lambda^2 f(x) \mbox{ for some } \lambda \in k^{\times} \} .
\end{equation}

As it is also noted in \cite[Section 2]{GS05} (without proof), any automorphism $\sigma$ of a hyperelliptic curve is represented by $\mathrm{diag}(\mu,1)\in \mathrm{GL}_2(k)$ for a primitive $\ell$-th root $\mu$ of unity with $\ell = \mathrm{ord}(\sigma)$ (the order of $\sigma$ as an element of $\overline{\rm Aut}(H)$), if $\ell$ is coprime to the characteristic of $k$.
An explicit proof of this fact is given in \cite{MK22}, and we here state assertions only:

\begin{prop}[{\cite[Proposition 2.2.2]{MK22}}]\label{prop:aut}
Let $H$ be a hyperelliptic curve of genus $g$ over an algebraically closed field $k$, and $\ell$ a positive integer coprime to $\mathrm{char}(k)$.
Assume that $\sigma$ has order $\ell$ in the reduced automorphism group of $H$.
Then there exists a hyperelliptic curve $H':y^2 = f(x)$ over $k$ and an isomorphism $\rho : H' \to H$ such that the automorphism $\rho^{-1}  \sigma \rho$ of $H'$ is represented by $(\mathrm{diag}(\mu,1),\mu')\in \mathrm{GL}_2(k) \times k^{\times}$, where $\mu$ is a primitive $\ell$-th root of $1$, and where $\mu'$ is an element satisfying $(\mu')^{\ell} = 1$ or $-1$.
We also have $\mu' = \pm \mu^{g+1}$ if $\mathrm{deg}(f) = 2g+2$, and $\mu' =\pm \sqrt{\mu^{2g+1}}$ if $\mathrm{deg}(f) = 2g+1$. 
Moreover, $\sigma$ is the hyperelliptic involution (i.e., $\ell=1$) if and only if $\mu=1$.
\end{prop}

\begin{lem}\label{lem:aut2}
Let $H$ and $\ell$ be as in Proposition \ref{prop:aut}.
If $H$ has an automorphism $\sigma$ of order $\ell \geq 2$ in the reduced automorphism group, then $\ell$ must divide $2g$, $2g+1$ or $2g+2$.
Moreover, $H$ is isomorphic to $y^2= f(x)$, where $f(x)$ is either of the following forms:
\begin{enumerate}
    \item[(1)] If $\ell$ divides $2g+2$, then
    \[
    f(x) =  x^{n \ell} + a_{(n-1)\ell} x^{(n-1)\ell}  + \cdots + a_{\ell} x^{\ell} + 1,
    \]
    where $n \ell = 2g+2$.
    \item[(2)] If $\ell$ divides $2g+1$, then
    \begin{align*}
        f(x) &= x ( x^{n \ell} + a_{(n-1)\ell+1} x^{(n-1)\ell}  + \cdots + a_{\ell +1} x^{\ell} + 1),
    \end{align*}
    where $n \ell = 2g+1$.
    \item[(3)] If $\ell$ divides $2g$, then
    \begin{align*}
        f(x) &= x ( x^{n \ell} + a_{(n-1)\ell+1} x^{(n-1)\ell}  + \cdots + a_{\ell +1} x^{\ell} + 1),
    \end{align*}
    where $n \ell = 2g$.
\end{enumerate}
\end{lem}

\begin{proof}
We may suppose from Proposition \ref{prop:aut} that $\sigma$ is represented by $(\mathrm{diag}(\mu,1), \mu') $, where $\mu$ and $\mu'$ are as in Proposition \ref{prop:aut}.
Let $y^2 = f(x)$ be an equation of $H$, where $f(x) = \sum_{i=0}^{d} a_i x^i$ with $d = 2g+1$ or $2g+2$, so that $(\mu')^{-2} = \mu^{-d}$.
One has that 
\begin{align*}
    (\mu')^{-2} f(\mu x) & = \mu^{d(\ell-1)} \sum_{i=0}^{d} a_{d-i} \mu^{d-i} x^{d-i}  = \sum_{i=0}^{d} a_{d-i} \mu^{d \ell -i} x^{d-i} = \sum_{i=0}^{d} a_{d-i} \mu^{-i} x^{d-i} \\
    & = a_{d} x^{d} + a_{d-1} \mu^{-1} x^{d-1} + \cdots + a_1 \mu^{-(d-1)} x + a_0 \mu^{-d},
\end{align*}
and hence it follows from $f(x) = (\mu')^{-2} f(\mu x)$ that $a_{d-i} = \mu^{-i} a_{d-i}$ for each $0 \leq i \leq d$.

If $\mu^{-(d-1)} \neq 1$ and $\mu^{-d} \neq 1$, then $a_1 = a_0 = 0$, namely $f(x)$ is divided by $x^2$, a contradiction.
Thus $\mu^{-(d-1)} = 1$ or $\mu^{-d} = 1$, and hence $\ell$ divides $d-1$ or $d$.
In any case, we obtain
\[
f(x) = a_{d} x^{d} + a_{d-\ell} x^{d -\ell} + a_{d-2\ell} x^{d-2\ell} + \cdots + a_{d-n\ell}x^{d-n\ell}
\]
where $n \ell = d-1$ or $d$.
More precisely, we have the following four (in fact three) cases:
\begin{itemize}
    \item[(i)] If $d= 2g+2$ and $\mu^{-d} = 1$, then $\ell$ divides $2g+2$, and
    \[
    f(x) = a_{n \ell} x^{n \ell} + a_{(n-1)\ell} x^{(n-1)\ell}  + \cdots + a_{\ell} x^{\ell} + a_{0}.
    \]
    By Remark \ref{rem:iso}, this case is nothing but the case (1).
    \item[(ii)] If $d= 2g+2$ and $\mu^{-(d-1)} = 1$, then $\ell$ divides $2g+1$, and
    \begin{align*}
        f(x) &= a_{n \ell +1} x^{n \ell + 1} + a_{(n-1)\ell + 1} x^{(n-1)\ell + 1}  + \cdots + a_{\ell + 1} x^{\ell + 1} + a_{1} x\\
        &= x (a_{n \ell+1} x^{n \ell} + a_{(n-1)\ell+1} x^{(n-1)\ell}  + \cdots + a_{\ell +1} x^{\ell} + a_{1})
    \end{align*}
    with $n \ell = 2g+1$.
    This is the case (2), by Remark \ref{rem:iso}.
    \item[(ii)'] If $d= 2g+1$ and $\mu^{-d} = 1$, then $\ell$ divides $2g+1$, and
    \[
    f(x) = a_{n \ell} x^{n \ell} + a_{(n-1)\ell} x^{(n-1)\ell}  + \cdots + a_{\ell} x^{\ell} + a_{0} 
    \]
    with $n \ell = 2g+1$.
    In this case, the hyperelliptic equation $y^2 = f(x)$ is transformed by $(x,y) \mapsto (1/x,y/x^{g+1})$ into
    \begin{align*}
    y^2 & = a_{n \ell} x^{(n \ell + 1) - n \ell} + \cdots + a_{i \ell} x^{(n \ell + 1) - i\ell} + \cdots  + a_{0} x^{n \ell + 1}\\
    &= a_0 x^{n \ell + 1}  + \cdots + a_{i \ell} x^{(n-i) \ell + 1} + \cdots  + a_{n \ell} x, 
    \end{align*}
    which is included in the case (ii).
    \item[(iii)] If $d= 2g+1$ and $\mu^{-(d-1)} = 1$, then $\ell$ divides $2g$, and
    \begin{align*}
        f(x) &= a_{n \ell +1} x^{n \ell + 1} + a_{(n-1)\ell + 1} x^{(n-1)\ell + 1}  + \cdots + a_{\ell + 1} x^{\ell + 1} + a_{1} x
    \end{align*}
    with $n \ell = 2g$.
    This is the case (3), by Remark \ref{rem:iso}.
\end{itemize}

\end{proof}

\begin{ex}\label{ex:aut}
For $p \geq 11$, hyperelliptic curves $H$ of genus $3$ over $k$ with extra automorphisms are (roughly) classified as follows, up to $k$-isomorphisms, in terms of reduced automorphism groups:
\begin{enumerate}
    \item ($\overline{\mathrm{Aut}}(H) \supset \mathbb{Z}_2$):
    $y^2 = x^{8} + a x^6 + b x^4 + c x^2 + 1$ or $y^2 = x^7 + a x^5 + b x^3 + x$ for $a, b, c \in k$.
     \item ($\overline{\mathrm{Aut}}(H) \supset \mathbb{Z}_3$):
    $y^2 = x^7 + a x^4 + x$ for $a \in k$.
    \item ($\overline{\mathrm{Aut}}(H) \supset \mathbb{Z}_4$): $y^2 = x^{8} + a x^4 +1$ for $a \in k$.
    \item ($\overline{\mathrm{Aut}}(H) \supset \mathbb{Z}_6$): $y^2 = x^{7} + x$.
    \item ($\overline{\mathrm{Aut}}(H) \supset \mathbb{Z}_7$): $y^2 = x^{8} + x$ (or $y^2 = x^7 + 1$).
    \item ($\overline{\mathrm{Aut}}(H) \supset \mathbb{Z}_{8}$): $y^2 = x^{8} + 1$.
\end{enumerate}
See e.g., \cite[Table 1]{GSS05} and \cite[Table 2]{LRS} for more precise classifications in the genus-$3$ hyperelliptic case.
\end{ex}

Recall from the begining of this section that $\overline{\mathrm{Aut}}(H)$ is embedded into $\mathrm{PGL}_2(k)$, and thus we obtain the classification of $\overline{\mathrm{Aut}}(H)$ as abstract finite groups in Theorem \ref{thm:RedAut} below, by applying the following classification (Theorem \ref{thm:Faber} and Proposition \ref{prop:Faber}) of subgroups of $\mathrm{PGL}_2(k)$ in characteristic $p$ (cf.\ the classification in Theorem \ref{thm:Faber} is well-known in characteristic zero):

\begin{thm}[\cite{Bea} and {\cite[Theorem C]{Faber}}]\label{thm:Faber}
Let $k$ be a separably closed field of characteristic $p$, and $G$ a finite subgroup of $\mathrm{PGL}_2(k)$ such that $p$ does not divide $\# G$.
Then, $G$ is isomorphic up to conjugation to one of the following subgroups:
\begin{enumerate}
    \item[(1)] $G = \langle \begin{pmatrix}
    \zeta & 0 \\
    0 & 1
    \end{pmatrix}\rangle$ for some primitive root $\zeta$ of unity in $k$; here $G$ is cyclic.
    \item[(2)] $G = \langle \begin{pmatrix}
    \zeta & 0 \\
    0 & 1
    \end{pmatrix}\rangle \rtimes \langle \begin{pmatrix}
    0 & 1 \\
    1 & 0
    \end{pmatrix}\rangle$ for some primitive root $\zeta$ of unity in $k$; here $G$ is dihedral.
    \item[(3)] $G = N \rtimes C \cong A_4$, where $N = \left\{ \begin{pmatrix}
    \pm 1 & 0 \\
    0 & 1
    \end{pmatrix}, \begin{pmatrix}
    0 & \pm 1 \\
    1 & 0
    \end{pmatrix}\right\}$
    and $C = \left\{ I_2, \begin{pmatrix}
    1 & i \\
    1 & -i
    \end{pmatrix}, \begin{pmatrix}
    1 & 1 \\
    -i & i
    \end{pmatrix}\right\}$.
    \item[(4)] $G = \langle T, \begin{pmatrix}
    i & 0\\
    0 & 1
    \end{pmatrix} \rangle \cong S_4$, where $T = N \rtimes C \cong A_4$ is the group in the case 3.
    \item[(5)] $G = \langle s, t \rangle$, where $s = \begin{pmatrix}
    \lambda & 0 \\
    0 & 1
    \end{pmatrix}$,
    $t = \begin{pmatrix}
    1 & 1-\lambda-\lambda^{-1} \\
    1 & -1
    \end{pmatrix}$ and $\lambda$ is any primitive fifth root of unity in $k$.
    These generators satisfy $s^5 = t^2 = (st)^3 = I_2$.
\end{enumerate}
\end{thm}

\begin{prop}[{\cite[Proposition 4.1]{Faber}}]\label{prop:Faber}
Let $k$ be an algebraically closed field of characteristic $p>0$.
If $G$ is a cyclic group of $\mathrm{PGL}_2(k)$ with $\# G = p$, then $G$ is conjugate to $\begin{pmatrix}
1 & \mathbb{F}_{p}\\
0 & 1
\end{pmatrix}$.
\end{prop}

\begin{thm}\label{thm:RedAut}
Let $H : y^2 = f(x)$ be a hyperelliptic curve of genus $g$ over a field of characteristic $p \neq 2$.
Then any subgroup of $\overline{\mathrm{Aut}}(H)$ (embedded in $\mathrm{PGL}_2(k)$) of order comprime to $p$ is conjugate to either of $\mathbb{Z}_n$, $D_n$, $A_4$, $S_4$ or $A_5$.
More precisely, there exists a hyperelliptic curve $H' : y^2 = f'(x)$ with an isomorphism $\rho : H' \to H$ such that the reduced automorphism group $\overline{G}_{f'}$ given in \eqref{eq:RedAut} at the beginning of this section is one of the five groups $G$ presented in Theorem \ref{thm:Faber}.
If $\overline{\rm Aut}(H)$ has an element whose order is divided by $p$, then $2g+1$ or $2g+2$ is divided by $p$.
\end{thm}

\begin{proof}
As for the first half part, it follows from Theorem \ref{thm:Faber} that there exists $P \in {\rm PGL}_2(k)$ such that $\overline{G}_{f} = P^{-1} G P$, where $G$ is one of the five groups given in Theorem \ref{thm:Faber}.
Hence, letting $H' : y^2 = f'(x) := P.f$ be a hyperelliptic curve with an isomorphism defined by $P$ from $H'$ to $H$, we have that $\overline{G}_{f'} = G$, namely there exists a one-to-one correspondence $\tau = \rho^{-1} \sigma \rho$ for $\tau \in \overline{\rm Aut}(H') = \overline{G}_{f'}$ and $\sigma \in \overline{\rm Aut}(H) = \overline{G}_{f}$, say
\[
\begin{CD}
H @>{\sigma}>> H \\
@A{\rho}AA @VV{\rho^{-1}}V   \\
H' @>{\tau}>> H' .
\end{CD}
\]
Assume that $\overline{\rm Aut}(H)$ has an element whose order is divided by $p$.
In this case, it follows from Cauchy's theorem that $\overline{\rm Aut}(H)$ has an element of order $p$.
Thus, by Proposition \ref{prop:Faber}, we may assume that $\overline{\rm Aut}(H)$ has a subgroup isomorphic to $\begin{pmatrix}
1 & \mathbb{F}_{p}\\
0 & 1
\end{pmatrix}$.
In particular, $H$ has an order-$p$ reduced automorphism $(x,y) \mapsto (x+1, \lambda y)$ represented by $\begin{pmatrix}
1 & 1\\
0 & 1
\end{pmatrix}$ and some $\lambda \in k^{\times}$.
From this, the number of ramification points is a multiple of $p$, as desired.
\end{proof}

\subsection{Some families of hyperelliptic curves whose reduced automorphism groups are dihedral}

In this subsection, we provide some explicit hyperellipitc curves whose reduced automorphism groups are $D_{g}$, $D_{g+1}$, $D_{2g}$ or $D_{2g+2}$.

First, we consider the cases of $\overline{\rm Aut}(H) \cong D_{2g}$ and $\overline{\rm Aut}(H) \cong D_{2g+2}$.

\begin{lem}\label{lem:D2g}
Assume that $g \geq 3$ and that $p$ does not divide $2g$, $2g+1$ nor $2g+2$.
Then we have the following:
\begin{enumerate}
    \item[(1)] A hyperelliptic curve of genus $g$ over $k$ defined by
    \begin{equation}\label{eq:D2g}
    y^2 = x^{2g+1} + x
    \end{equation}
    has the reduced automotphism group isomorphic to $D_{2g}$.
    
    Conversely, any hyperelliptic curve $H$ of genus $g$ over $k$ such that $\overline{\rm Aut}(H) $ contains a subgroup isomorphic to $\mathbb{Z}_{2g}$ is birational to \eqref{eq:D2g}.
    
    \item[(2)] A hyperelliptic curve of genus $g$ over $k$ defined by
    \begin{equation}\label{eq:D2g2}
    y^2 = x^{2g+2} + 1
    \end{equation}
    has the reduced automotphism group isomorphic to $D_{2g+2}$.
    
    Conversely, any hyperelliptic curve $H$ of genus $g$ over $k$ such that $\overline{\rm Aut}(H) $ contains a subgroup isomorphic to $\mathbb{Z}_{2g+2}$ is birational to \eqref{eq:D2g2}.
\end{enumerate}
\end{lem}

\begin{proof}
We here prove the case (1) only, since the case (2) is proved similarly.
Let $H: y^2 = x^{2g+1} + x$, and $\zeta$ a primitive $2g$-th root of unity in $k$.
Since $H$ has two automorphisms $\sigma : (x,y) \mapsto (\zeta x,y)$ of order $2g$ and $\tau : (x,y) \mapsto \left(\frac{1}{x},\frac{y}{x^{g+1}} \right)$ of order $2$ with $(\sigma \tau)^2 = \iota_H$, we have $\overline{\mathrm{Aut}}(H)\supset D_{2g}$.
Here $\overline{\mathrm{Aut}}(H)$ is either of the five types in Theorem \ref{thm:RedAut}, but clearly $\mathbb{Z}_n$ is impossible.
Recall from Lemma \ref{lem:aut2} that the order of any element in $\overline{\rm Aut}(H)$ is upper-bounded by $2g+2$, and thus $\overline{\mathrm{Aut}}(H) = D_n$ with $n > 2g+2$ does not hold.
As for $A_5$, the dihedral groups contained in it are $D_2 = V_4$, $D_{3}=S_3$ and $D_5$, and thus $\overline{\mathrm{Aut}}(H)$ is not isomorphic to $A_5$.
Similarly, $A_4$ is also not possible.
The converses follow immediately from Lemma \ref{lem:aut2}.
\end{proof}

The next cases which we consider are the cases $\overline{\rm Aut}(H) \supset D_{g}$ and $\overline{\rm Aut}(H) \supset D_{g+1}$.

\begin{lem}\label{lem:Dg}
Assume that $g \geq 3$ and that $p$ does not divide $g$ nor $g+1$.
Then we have the following:
\begin{enumerate}
    \item[(1)] A hyperelliptic curve of genus $g$ over $k$ defined by
    \begin{equation}\label{eq:Dg}
    y^2 = x^{2g+1} + a x^{g+1} + x
    \end{equation}
    for some $a \in k$ has the reduced automorphism group containing $D_g$.
    
    Conversely, any hyperelliptic curve $H$ of genus $g$ over $k$ such that $\overline{\rm Aut}(H) $ contains a subgroup isomorphic to $\mathbb{Z}_{g}$ is birational to \eqref{eq:Dg}.
    
    \item[(2)] A hyperelliptic curve of genus $g$ over $k$ defined by
    \begin{equation}\label{eq:Dg1}
    y^2 = x^{2g+2} + a x^{g+1} + 1
    \end{equation}
    for some $a \in k$ has the reduced automorphism group containing $D_{g+1}$
    
    Conversely, any hyperelliptic curve $H$ of genus $g$ over $k$ such that $\overline{\rm Aut}(H) $ contains a subgroup isomorphic to $\mathbb{Z}_{g+1}$ is birational to \eqref{eq:Dg1}.
\end{enumerate}
In the case (1) (resp.\ (2)), $\overline{\rm Aut}(H) $ is isomorphic to $D_g$, $D_{2g}$ (resp.\ $D_{g+1}$, $D_{2g+2}$), $A_4$, $S_4$ or $A_5$.

\end{lem}

\begin{proof}
The reduced automotphism group of $y^2 = x^{2g+1} + a x^{g+1} + x$ (resp.\ $y^2 = x^{2g+2} + a x^{g+1} + 1$) has elements $\sigma : (x,y) \mapsto (\zeta x,y)$ of order $g$ (resp.\ $g+1$) and $\tau : (x,y) \mapsto \left(\frac{1}{x},\frac{y}{x^{g+1}} \right)$ of order $2$ with $(\sigma \tau)^2 = \iota_H$.
Thus, we have $\overline{\mathrm{Aut}}(H)\supset D_{g}$ (resp.\ $\overline{\mathrm{Aut}}(H)\supset D_{g+1}$).

The converses follow immediately from Lemma \ref{lem:aut2}.
\end{proof}





\subsection{Classification of automorphism groups and defining equations when $g=4$}

Assume $p \geq 7$.
As in Example \ref{ex:aut} for the genus-$3$ case, it follows from Lemma \ref{lem:aut2} together with Theorem \ref{thm:RedAut} that hyperelliptic curves $H$ of genus $4$ over $k$ with $\# \overline{\mathrm{Aut}}(H) \geq 2$ are (roughly) classified up to $k$-isomorphisms:
\begin{enumerate}
    \item ($\overline{\mathrm{Aut}}(H) \supset \mathbb{Z}_2$):
    $y^2 = x^{10} + a x^8 + b x^6 + c x^4 + d x^2 + 1$ for $a, b, c, d \in k$,
    \item ($\overline{\mathrm{Aut}}(H) \supset \mathbb{Z}_2$):
    $y^2 = x^{9} + a x^7 + b x^5 + c x^3 + x$ for $a, b, c \in k$,
    \item ($\overline{\mathrm{Aut}}(H) \supset \mathbb{Z}_3$): $y^2 = x^{10} + a x^7 + b x^4 + x$ for $a, b \in k$,
    \item ($\overline{\mathrm{Aut}}(H) \supset \mathbb{Z}_4$): $y^2 = x^{9} + a x^5 + x$ for $a \in k$,
    \item ($\overline{\mathrm{Aut}}(H) \supset \mathbb{Z}_5$): $y^2 = x^{10} + a x^5 + 1$ for $a \in k$,
    \item ($\overline{\mathrm{Aut}}(H) \supset \mathbb{Z}_8$): $y^2 = x^{9} + x$,
    \item ($\overline{\mathrm{Aut}}(H) \supset \mathbb{Z}_9$): $y^2 = x^{10} + x$,
    \item ($\overline{\mathrm{Aut}}(H) \supset \mathbb{Z}_{10}$): $y^2 = x^{10} + 1$.
\end{enumerate}
Moreover, since $D_n$ has an element of order $n$, we have that $\overline{\mathrm{Aut}}(H) $ is isomorphic to one of the following:
$\{ 0 \}$, $\mathbb{Z}_n$, $D_n$ for $n= 2,3,4,5,8,9,10$, $A_4$, $S_4$ or $A_5$.

In the following, we provide a more precised classification: The possible types of $\overline{\mathrm{Aut}}(H) $ and corresponding equations $y^2 = f(x)$, in order to prove Theorem \ref{thm:app}.
Note that $D_2$ and $D_3$ are nothing but the Klein 4-group $V_4 := \mathbb{Z}_2 \times \mathbb{Z}_2$ and the symmetric group of degree three $S_3$ respectively.
We start with considering the cases where $\overline{\rm Aut}(H)\cong \mathbb{Z}_2$ or $\overline{\rm Aut}(H) \cong D_2=V_4$.

\begin{lem}\label{lem:Z2V4}
Let $H$ be a hyperelliptic curve of genus $4$ over an algebraically closed field $k$ of characteristic $p = 0$ or $p \geq 3$.
Assume that $\overline{\rm Aut}(H) \supset \mathbb{Z}_2$, i.e., the reduced automorphism group has an element $\sigma$ of order $2$;
$\sigma^2 = id_{H}$ or $\iota_H$ in $\mathrm{Aut}(H)$, equivalently ${\rm Aut}(H) \supset V_4$ or $\mathbb{Z}_4$.
In each case, $H$ is isomorphic to $y^2 = f(x)$ given as follows:
\begin{enumerate}
    \item[{\bf 2-1}.] $({\rm Aut}(H) \supset V_4)$ $y^2 = x^{10} + a x^8 + b x^6 + c x^4 + d x^2 + 1$ for some $a, b, c, d \in k$.
    \item[{\bf 2-2}.] $({\rm Aut}(H) \supset \mathbb{Z}_4)$ $y^2 = x^{9} + a x^7 + b x^5 + c x^3 + x$ for some $a, b, c\in k$.
\end{enumerate}
Moreover, if $\overline{\rm Aut}(H) \supset V_4$, then $H$ is isomorphic to $y^2 = f(x)$ given as follows:
\begin{enumerate}
    \item[{\bf 4-1}.] $y^2 = (x^2+1)(x^4 + A x^2 + 1) (x^4 + B x^2 + 1) = x^{10} + a x^8 + b x^6 + b x^4 + a x^2 + 1$ for some $A, B, a, b \in k$, so that ${\rm Aut}(H) \supset D_4$.
    \item[{\bf 4-2}.] $y^2 = x(x^4 + A x^2 + 1) (x^4 + B x^2 + 1) = x^{9} + a x^7 + b x^5 + a x^3 + x $ for some $A, B, a, b \in k$, so that ${\rm Aut}(H) \supset D_4$.
    \item[{\bf 4-3}.] $y^2 = x(x^4 -1) (x^4 + A x^2 + 1) = x^{9} + A x^7 - A x^3 - x $ for some $A\in k$, so that ${\rm Aut}(H) \supset Q_8$.
\end{enumerate}
\end{lem}

\begin{proof}
The first half part is proved similarly to the proof of (1) $\Rightarrow$ (2) of Lemma \ref{lem:isom3}, based on Proposition \ref{prop:aut}.
To prove the second half part, we suppose that $\overline{\rm Aut}(H) \supset V_4$, and let $y^2 = f(x)$ be an equation defining $H$.
In this case, we may assume from Theorems \ref{thm:Faber} and \ref{thm:RedAut} that $H$ has automorphisms $(x,y) \mapsto (-x,\lambda y)$, $(x,y) \mapsto \left( \frac{1}{x}, \frac{\mu y}{x^5} \right)$ and $(x,y) \mapsto \left( \frac{-1}{x}, \frac{\lambda \mu y}{x^5} \right)$ with $\lambda^2 = \pm 1$ and $\mu^2 = \pm 1$.
Thus, if $(\alpha, 0)$ is a ramification point of $H$, then so are $(\pm \alpha, 0)$ and $\left( \pm \frac{1}{\alpha}, 0 \right)$.
This implies that $f(x)$ is a constant multiple of $f_1(x) := (x^2 \pm 1) g_A(x) g_B(x)$, $f_2 (x):=x g_A(x) g_B(x)$ or $f_3(x):=x (x^4-1)g_B(x)$ with $g_A(x) := (x^4 + A x^2 + 1)$ for $A = \alpha^2 + \frac{1}{\alpha^2}$ and $g_{B}(x) = x^4 + B x^2 + 1$ for $B = \beta^2 + \frac{1}{\beta^2}$, where $(\alpha, 0)$ and $(\beta, 0)$ are ramification points of $H$ with $\alpha \neq 0$ and $\beta \neq 0$.
Therefore, $H$ is isomorphic to $y^2=f_j(x)$ for some $j$, where $y^2=(x^2+1)g_A(x)g_B(x)$ and $y^2 = (x^2-1)g_{-A}(x)g_{-B}(x)$ are isomorphic via $(x,y) \mapsto (i x, i y)$.
\end{proof}

Next we consider the cases $A_4$, $S_4$ and $A_5$.

\begin{lem}\label{lem:A4orS4}
If the reduced automorphism group of a hyperelliptic curve $H$ of genus $4$ over $k$ of characteristic $p \notin \{ 2, 3 \}$ contains a subgroup isomorphic to $A_4$, then it is isomorphic to $A_4$ itself, $S_4$, or its order is divided by $12p$.
Hence if $p \geq 7$, then $\overline{\rm Aut}(H) \cong A_4$ or $S_4$.
\end{lem}

\begin{proof}
If $\overline{\rm Aut}(H)$ has order coprime to $p$, then it is isomorphic to $\{ 0 \}$, $\mathbb{Z}_n$, $D_n$ for $n= 2,3,4,5,8,9,10$, $A_4$, $S_4$ or $A_5$;
among them, only $A_4$ and $S_4$ are possible.
Otherwise, clearly the order of $\overline{\rm Aut}(H)$ is divided by $12 p$.
It follows from Theorem \ref{thm:RedAut} that we have $\overline{\rm Aut}(H) \cong A_4$ or $S_4$ if $p \geq 7$.
\end{proof}

\begin{lem}\label{lem:A4S4}
Let $H$ be a hyperelliptic curve $H$ of genus $4$ over $k$ of characteristic $p \geq 7$.
If $\overline{\rm Aut}(H) $ contains a subgroup isomorphic to $A_4$, then it is isomorphic to $A_4$, and is birational to 
\begin{equation}\label{eq:A4}
y^2 = x (x^4-1)(x^4 + 2 i \sqrt{3} x^2 + 1).
\end{equation}
Conversely, a hyperelliptic curve given by \eqref{eq:A4} has the reduced automotphism group $A_4$ if $p \geq 5$, where $i$ is an element in $\mathbb{F}_{p^2} \subset k$ satisfying $i^2= -1$.
(Hence, since $S_4 \supset A_4$, there is no hyperelliptic curve $H$ of genus $4$ over $k$ such that $\overline{\rm Aut}(H) \cong S_4$.)
\end{lem}

\begin{proof}
Lemma \ref{lem:A4orS4} implies that $\overline{\rm Aut}(H) \cong A_4$ or $S_4$.
Also by Theorem \ref{thm:Faber}, we may assume that $\overline{\rm Aut}(H)$ is generated by the two automorphisms $\sigma : x \mapsto -x$ and $\tau : x \mapsto \frac{x+i}{x-i}$ if $\overline{\rm Aut}(H) \cong A_4$, and $\sigma$, $\tau$ and $\rho : x \mapsto i x$ if $\overline{\rm Aut}(H) \cong S_4$, where $i$ is an element in $\mathbb{F}_{p^2}$ satisfying $i^2= -1$.
In any case, $H$ has automorphisms $\sigma$ and $x \mapsto 1/x$.
Thus, similarly to the proof of Lemma \ref{lem:Z2V4}, we may also suppose that $H$ is isomorphic to $y^2=f(x)$, where $f(x)$ is either of the following forms: 
\begin{itemize}
    \item $f(x) = (x^2+1)(x^4+Ax^2 + 1)(x^4+Bx^2+1)= x^{10} + a x^{8} + b x^6 + b x^4 + a x^2 + 1$,
    \item $f(x) = x (x^4+Ax^2 + 1)(x^4+Bx^2+1) =  x^{9} + a x^{7} + b x^5 + a x^3 +x$,
    \item $f(x) = x (x^4 - 1) (x^4 + A x^2 + 1) = x^9 + A x^7 - A x^3 - x$.
\end{itemize}
We claim that the first case is not possible.
Indeed, since $H$ has an automorphism $\tau$, we also have 
\begin{equation}\label{eq:A4proof}
    (x - i)^{10} f\left( \frac{x +i}{x - i} \right) =  \lambda^2 f(x)
\end{equation}
for some $\lambda \in k^{\times}$.
In the first case, the coefficient of $x^{10}$ and the constant term in the left-hand side are $2+ 2 a + 2b$ and $-(2+ 2 a + 2b)$, while those in the right-hand side are both $\lambda^2$.
This means that $\lambda = 0$, a contradtiction.

In the second and third cases, $H$ has $0 \in k$ as a ramification point, and one can check (cf.\ \cite[Remark 4.3]{Shaska04}) that the orbit of $0$ by the action of $\overline{\rm Aut}(H)$ is $\{ \infty, 0, \pm 1, \pm i\}$.
Therefore, $f(x)$ is divided by $x^4-1$, and thus this is the third case.
Comparing the coefficients of $x^{9}$ (resp.\ $x^7$) in the both sides of \eqref{eq:A4proof}, one has the equation $8 i A  + 16 i =  \lambda^2$ (resp.\ $16 i A - 96 i = \lambda^2 A$), and therefore $A = \pm 2 i \sqrt{3}$ by $p \neq 3$.
Note that $y^2 = x (x^4 - 1) (x^4 + A x^2 + 1)$ and $y^2 = x (x^4 - 1) (x^4 - A x^2 + 1)$ are isomorphic via $x \mapsto i x$, where $\rho \notin \overline{\rm Aut}(H)$, so that $\overline{\mathrm{Aut}}$ is not isomorphic to $S_4$, as desired.

For the converse, it is straightforward to see that the reduced automorphism group of \eqref{eq:A4} surely has a subgroup isomorphic to $A_4$, and thus it is isomorphic to $A_4$ as proved as above.
The case $p=5$ is proved with Magma's built-in function ``AutomorphismGroupOfHyperellipticCurve''.
\end{proof}

\begin{lem}\label{lem:A5}
If $p\geq 7$, there is no hyperelliptic curve $H$ of genus $4$ over $k$ such that $\overline{\mathrm{Aut}}(H) \cong A_5$.
\end{lem}

\begin{proof}
Suppose for a contradiction that there exists a hyperelliptic curve $H$ of genus $4$ over $k$ with $\overline{\mathrm{Aut}}(H) \cong A_5$.
By Theorems \ref{thm:Faber} and \ref{thm:RedAut}, we may assume that $\overline{\mathrm{Aut}}(H)$ is generated by $(x,y) \mapsto (\zeta x, \lambda y) $ and $(x, y) \mapsto \left( \frac{x + (1-\zeta - \zeta^{-1})}{x - 1}, \mu y \right)$ for some $\lambda, \mu \in k^{\times}$, where $\zeta$ is a primitive $5$-th root of unity in $k^{\times}$.
It follows from $f (\zeta x) = \lambda^2 f(x)$ that $f(x) = a x^{10} + b x^5 + c$ and $\lambda^2 = 1$ for some $a,b,c \in k$.
By comparing the coefficients of $x^9$, $x^8$ and $x^7$ in the both sides of
\[
(x-1)^{10} f \left( \frac{x + (1-\zeta - \zeta^{-1})}{x - 1} \right) =  \mu^2 f(x),
\]
we have the following linear system:
\begin{equation}
\begin{pmatrix}
10 \zeta^3 + 10 \zeta^2 + 20  &  5 \zeta^3 + 5 \zeta^2 + 5  &  -10  \\
135 \zeta^3 + 135 \zeta^2 + 225 & 5 \zeta^3 + 5 \zeta^2 + 10  &  45 \\
960 \zeta^3 + 960 \zeta^2 + 1560 & -20 \zeta^3 -20 \zeta^2 -30  &  - 120 
\end{pmatrix}
\begin{pmatrix}
a \\
b \\
c
\end{pmatrix}
=
\begin{pmatrix}
0 \\
0 \\
0
\end{pmatrix}.
\end{equation}
The determinant of the coefficient matrix is computed as $2^2 \cdot 3 \cdot 5^5 (2 \cdot 3^2 \zeta^3 + 2 \cdot 3^2 \zeta^2 + 29) \neq 0$, which is not equal to zero by $p \geq 7$.
Thus $(a,b,c) = (0,0,0)$, which is a contradiction.
\end{proof}

Here, let us investigate the cases $\mathbb{Z}_n$ and $D_n$ with $n = 3$ and $9$.

\begin{lem}\label{lem:noD3D9}
We have the followng:
\begin{enumerate}
    \item[(1)] There is no hyperelliptic curve $H$ of genus $4$ over $k$ such that $\overline{\mathrm{Aut}}(H) \cong D_3,D_9$.
    \item[(2)] For $p \geq 5$, a hyperelliptic curve of genus $4$ over $k$ defined by $H:y^2 = x^{10}+x$ has the reduced automorphism group $\mathbb{Z}_9$.
\end{enumerate}

\end{lem}

\begin{proof}
\begin{enumerate}
    \item Let $H$ be a a hyperelliptic curve $H$ of genus $4$ over $k$.
    It suffices to show that $\overline{\mathrm{Aut}}(H)$ has no subgroup isomorphic to $D_3$, since $D_9$ contains $D_3$ as a subgroup.
    Suppose for a contradiction that $\overline{\mathrm{Aut}}(H) \supset D_3$.
    By Theorems \ref{thm:Faber} and \ref{thm:RedAut}, we may assume that $\overline{\mathrm{Aut}}(H)$ has a subgroup generated by $(x,y) \mapsto (\zeta x, \lambda y) $ and $(x, y) \mapsto \left( \frac{1}{x}, \mu y \right)$ for some $\lambda, \mu \in k^{\times}$, where $\zeta$ is a primitive $3$-rd root of unity in $k^{\times}$.
    It follows from $f (\zeta x) = \lambda^2 f(x)$ that we have $f(x) = a x^{10} + b x^7 + c x^4 + d x$ for some $a,b,c \in k$, but this contradicts that $(x, y) \mapsto \left( \frac{1}{x}, \mu y \right)$ is an automorphism of $H$.
    \item Clearly, $\overline{\rm Aut}(H) $ has an element $(x,y) \mapsto (\zeta x, \sqrt{\zeta} y) $ of order $9$, where $\zeta$ is a primitive $9$-th root of unity in $k$, and therefore $\overline{\rm Aut}(H) \supset \mathbb{Z}_9$. 
    If $p \geq 7$, it follows from Theorem \ref{thm:RedAut} that $\overline{\mathrm{Aut}}(H)$ is either of the five types in the theorem, but clearly $A_4$, $S_4$ and $A_5$ are impossible.
    As for the cases $\mathbb{Z}_n$ and $D_n$, only $\overline{\mathrm{Aut}}(H) \cong D_9$ or $\overline{\mathrm{Aut}}(H) \cong \mathbb{Z}_9$ is possible since the order of any element in $\overline{\mathrm{Aut}}(H)$ divides $8$, $9$ or $10$.
    By the assertion (1), we have $\overline{\mathrm{Aut}}(H) \cong \mathbb{Z}_9$, as desired.
    The case $p=5$ is proved with Magma's built-in function ``AutomorphismGroupOfHyperellipticCurve''.
\end{enumerate}
\end{proof}

\begin{lem}\label{lem:C3}
Let $H$ be a hyperelliptic curve of genus $4$ over $k$ of characteristic $p \geq 7$.
If the reduced automorphism group contains a subgroup isomorphic to $\mathbb{Z}_3$, then it is isomorphic to $\mathbb{Z}_3$, $\mathbb{Z}_9$ or $A_4$.
Moreover, in each case, $H$ is isomorphic to $y^2 = f(x)$ given as follows:
\begin{itemize}
    \item $(\overline{\rm Aut}(H)) \cong \mathbb{Z}_3)$ $y^2 = x^{10} + a x^7 + b x^4 + x$ for some $a, b \in k$ with $(a,b) \neq (0,0)$.
    \item $(\overline{\rm Aut}(H)) \cong \mathbb{Z}_9)$ $y^2 = x^{10} + x$.
    \item $(\overline{\rm Aut}(H)) \cong A_4)$ $y^2 = x (x^4-1)(x^4 + 2 i \sqrt{3} x^2 + 1)$ as in \eqref{eq:A4}.
\end{itemize}
\end{lem}

\begin{proof}
By Theorem \ref{thm:RedAut}, $\overline{\mathrm{Aut}}(H)$ is either of the 5 types in the theorem, but $S_4$ and $A_5$ are impossible by Lemmas \ref{lem:A4S4} and \ref{lem:A5}.
As $\mathbb{Z}_n$ and $D_n$, since the order of any element in $\overline{\mathrm{Aut}}(H)$ divides $8$, $9$ or $10$, one has that $\overline{\mathrm{Aut}}(H)$ is isomorphic to $\mathbb{Z}_3$, $\mathbb{Z}_9$, $D_3$ or $D_9$, but $D_3$ and $D_9$ do not occur by Lemma \ref{lem:noD3D9} (1).
The statements on equations hold by Lemmas \ref{lem:aut2} and \ref{lem:A4S4}.
\end{proof}

As for the remaining cases $\mathbb{Z}_n$ and $D_n$ with $n = 4$, $5$, $8$ and $10$, it follows from Lemmas \ref{lem:D2g} and \ref{lem:Dg} that we have the following:

\begin{lem}\label{lem:D45810}
Assume that $p \geq 7$.
Then there is no hyperelliptic curve of genus $4$ over $k$ satisfying $\overline{\mathrm{Aut}}(H) \cong \mathbb{Z}_n$ with $n=4$, $5$, $8$ or $10$.
Moreover, we have the following:
\begin{enumerate}
    \item[(1)] If $\overline{\mathrm{Aut}}(H) \cong {D}_{4}$ (resp.\ $D_8$), then $H$ is isomorphic to $y^2 = x^{9} + a x^{5} + x$ (resp.\ $y^2 = x^9 + x$) for some $a \in k^{\times}$.
   \item[(2)] If $\overline{\mathrm{Aut}}(H) \cong {D}_{5}$ (resp.\ $D_{10}$), then $H$ is isomorphic to $y^2 = x^{10} + a x^{5} + 1$ (resp.\ $y^2 = x^{10} + 1$) for some $a \in k^{\times}$.
\end{enumerate}
\end{lem}


We finally prove Theorem \ref{thm:app}.\\

\noindent {\it Proof of Theorem \ref{thm:app}.}
As it was noted at the beginning of this section, it follows from Lemma \ref{lem:aut2} together with Theorem \ref{thm:RedAut} that $\overline{\mathrm{Aut}}(H) $ is isomorphic to one of the following:
$\{ 0 \}$, $\mathbb{Z}_n$, $D_n$ for $n= 2,3,4,5,8,9,10$, $A_4$, $S_4$ or $A_5$.
By Lemmas \ref{lem:A4S4} -- \ref{lem:D45810}, the possible cases are $\mathbb{Z}_n$ for $n= 1, 2, 3, 9$, $D_n$ for $n=2,4,5,8,10$, or $A_4$, and the corresponding equations are those provided in Table \ref{table:aut}.
For each case, the finite group isomorphic to $\mathrm{Aut}(H)$ is determined by straightforward computation together with Lemma \ref{lem:Z2V4}, based on $\overline{\mathrm{Aut}}(H) = \mathrm{Aut}(H)/ \langle \iota_H \rangle $, where $H$ is the hyperelliptic involution.
The containment relationships in Figure \ref{fig:aut} follows from group theory.\qed